\documentclass[reqno]{amsart}

\usepackage{graphicx}
\usepackage[dvips]{epsfig}
\usepackage{bm}
\usepackage{mathrsfs}
\usepackage{psfrag} 

 \makeatletter
\@namedef{subjclassname@2020}{%
  \textup{2020} Mathematics Subject Classification}
\makeatother

\newtheorem{theorem}{Theorem}[section]
\newtheorem{lemma}{Lemma}[section]
\newtheorem{proposition}[lemma]{Proposition}
\newtheorem{corollary}[lemma]{Corollary}
\newtheorem{definition}[lemma]{Definition}
\newtheorem{remark}[lemma]{Remark}
\newtheorem{problem}{{\bf{Problem}}}[section]

\numberwithin{figure}{section}

\numberwithin{equation}{section}

\allowdisplaybreaks
\begin{document}

\title[Transonic shocks]{Transonic shocks for three-dimensional axisymmetric flows in divergent nozzles}

\author{Hyangdong Park}
\address{School of Mathematics, Korea Institute for Advanced Study (KIAS), 85 Hoegiro, Dongdaemun-gu, Seoul, 02455, Republic of Korea}
\email{hyangdong@kias.re.kr}



\date{\today}

\begin{abstract}
We prove the stability of three-dimensional axisymmetric solutions to the steady Euler system with transonic shocks in divergent nozzles under perturbations of the exit pressure and the supersonic solution in the upstream region. We first derive a free boundary problem with the newly introduced formulation of the Euler system for three-dimensional axisymmetric flows in divergent nozzles via the method of Helmholtz decomposition. We then construct an iteration scheme and use the Schauder fixed point theorem and weak implicit function theorem to solve the problem. 
\end{abstract}

\maketitle

\newcommand {\gam}{\gamma}
\newcommand \Gam{\Gamma}
\newcommand \trho{\tilde{\rho}}
\newcommand \vphi{\varphi}
\newcommand \ol{\overline}
\newcommand \Om{\Omega}
\newcommand \om{\omega}
\newcommand \der{\partial}
\newcommand \tx{\text}
\newcommand \mcl{\mathcal}
\newcommand \eps{\varepsilon}
\newcommand \mfrak{\mathfrak}
\newcommand \Gamw{\Gamma_{{\rm w}}}
\newcommand \rx{{\rm x}}



\section{Introduction}
In $\mathbb{R}^3$, the steady  inviscid compressible flows of ideal polytropic gas  are governed by the steady Euler system as follows: 
\begin{equation}\label{Euler-system-B}
\left\{\begin{split}
&\mbox{div}(\rho{\bf u})=0,\\
&\mbox{div}(\rho{\bf u}\otimes{\bf u})+\nabla p={\bf 0},\\
&\mbox{div}\left(\rho{\bf u}B\right)=0,
\end{split}\right.
\end{equation}
where $\rho=\rho({\bf x})$, ${\bf u}=(u_1,u_2,u_3)({\bf x})$, $p=p({\bf x})$, and $B=\left(\frac{\lvert{\bf u}\rvert^2}{2}+\frac{\gamma p}{(\gamma-1)\rho}\right)({\bf x})$  denote the density, velocity, pressure, and  \textit{the Bernoulli invariant} of the flow for \textit{the adiabatic exponent} $\gamma>1$, respectively, at ${\bf x}=(x_1,x_2,x_3)\in\mathbb{R}^3$. 
In this paper, we study the stability of transonic shocks for the steady Euler system \eqref{Euler-system-B}.

Let $\Omega\subset\mathbb{R}^3$ be an open connected set, and let $\Gamma$ be a non-self-intersecting $C^1$ surface dividing $\Omega$ into two disjoint open subsets $\Omega^{\pm}$ such that $\Omega=\Omega^-\cup\Gamma\cup\Omega^+$.
\begin{definition}
Define ${\bf U}=(\rho,{\bf u},p)\in[L^{\infty}_{\rm{loc}}(\Om)\cap C^1_{\rm{loc}}(\Om^{\pm})\cap C^0_{\rm{loc}}(\Om^{\pm}\cup\Gamma)]^5$ to be a weak solution to \eqref{Euler-system-B} in $\Omega$ if the following properties are satisfied:
For any test function $\xi\in C_0^{\infty}(\Om)$ and $j=1,2,3$,
\begin{equation}\label{we-eq}
\int_{\Om}\rho {\bf u}\cdot\nabla\xi d{\bf x}=\int_{\Om}(\rho u_j{\bf u}+p{\bf e}_j)\cdot\nabla\xi d{\bf x}=\int_{\Om}\rho {\bf u} B\cdot\nabla\xi d{\bf x}=0,
\end{equation}
where each ${\bf e}_j$ is the unit vector in the $x_j$-direction.
\end{definition}
By the integration by parts, ${\bf U}$ satisfies \eqref{we-eq} if and only if 
\begingroup
\allowdisplaybreaks
\begin{align}
(w_1)\,\,&\label{w1-con}{\bf U} \mbox{ is a classical solution to \eqref{Euler-system-B} in }\Om^{\pm};\\
(w_2)\,\,&\nonumber{\bf U}\mbox{ satisfies the Rankine-Hugoniot conditions}\\
&\label{w2-con}(RH_1)\quad [\rho{\bf u}\cdot{\bf n}]_{\Gamma}=0,\quad [\rho({\bf u}\cdot{\bf n})^2+p]_{\Gamma}=0,\quad [\rho{\bf u}\cdot{\bf n}B]_{\Gamma}=0,\\
&\nonumber(RH_2)\quad \rho({\bf u}\cdot{\bf n})[{\bf u}\cdot{\bm\tau}_k]_{\Gamma}=0\quad\mbox{for all }k=1,2,
\end{align}
\endgroup
where $[\,\cdot\,]_{\Gamma}$ is defined by $[F({\bf x})]_{\Gamma}:=\left.F({\bf x})\right|_{\overline{\Om^-}}-\left.F({\bf x})\right|_{\overline{\Om^+}}$ for ${\bf x}\in\Gamma$, ${\bf n}$ is a unit normal vector field on $\Gamma$, and ${\bm \tau}_k$ $(k=1,2)$ are unit tangent vector fields on $\Gamma$ such that they are linearly independent at each point on $\Gamma$.

Suppose that $\rho>0$ in ${\Omega}$.
Then the condition $(RH_2)$ is satisfied if either ${\bf u}\cdot{\bf n}=0$ on $\Gamma$, or $[{\bf u}\cdot{\bm \tau}_k]_{\Gamma}=0$ for all $k=1,2$. 
$\Gamma$ is called a contact discontinuity if ${\bf u}\cdot{\bf n}=0$ on $\Gamma$.
If $[{\bf u}\cdot{\bm \tau}_k]_{\Gamma}=0$ for all $k=1,2$, $\Gamma$ is called a shock.
In this paper, we focus on shocks. For the study  of contact discontinuities, one may refer to \cite{bae2019contact, bae2019contact3D,Huang:2021aa} and the references cited therein.

We define a weak solution to \eqref{Euler-system-B} with a transonic shock as follows:

\begin{definition}{\cite[Definition 1.2]{park2020transonic}}\label{Def-shock}
Define ${\bf U}=(\rho,{\bf u},p)\in[L^{\infty}_{\rm{loc}}(\Om)\cap C^1_{\rm{loc}}(\Om^{\pm})\cap C^0_{\rm{loc}}(\Om^{\pm}\cup\Gamma)]^5$ to be a weak solution to \eqref{Euler-system-B} in $\Omega$ with a transonic shock $\Gamma$ if the following properties are satisfied:
\begin{itemize}
\item[(a)] {\rm (Shock)} $\Gamma$ is a non-self-intersecting $C^1$ surface dividing $\Omega$ into two open subsets $\Om^{\pm}$ such that $\Om=\Om^-\cup\Gamma\cup\Omega^+$; 
\item[(b)] {\rm (Classical solution)} ${\bf U}$ satisfies $(w_1)$ in \eqref{w1-con};

\item[(c)] {\rm (Positivity of density)} $\rho>0$ in $\overline{\Om}$;

\item[(d)] {\rm (Rankine-Hugoniot conditions)} ${\bf U}$ satisfies $(RH_1)$ in \eqref{w2-con} and $[{\bf u}\cdot{\bm\tau}_k]_{\Gamma}=0$ for all $k=1,2$;
\item[(e)] {\rm(Transonic)} The solution is supersonic in $\Om^-$ and subsonic in $\Om^+$, i.e., $|{\bf u}|>c(\rho,p)$ in $\Om^-$ and $|{\bf u}|<c(\rho,p)$  in $\Om^+$ for the sound speed $c(\rho,p):=\sqrt{\frac{\gamma p}{\rho}}$;
\item[(f)] {\rm(Admissibility)} ${\bf u}|_{\overline{\Om^-}\cap\Gamma}\cdot{\bf n}>{\bf u}|_{\overline{\Om^+}\cap\Gamma}\cdot{\bf n}>0$ for the unit normal vector field ${\bf n}$ on $\Gamma$ pointing toward $\Om^+$.
\end{itemize}
\end{definition}
The goal of this work is to establish the existence and stability of weak solutions to the steady Euler system \eqref{Euler-system-B} with  transonic shocks in a three-dimensional divergent nozzle (Figure \ref{figure1}).
\begin{figure}[!h]
\centering
\includegraphics[scale=0.55]{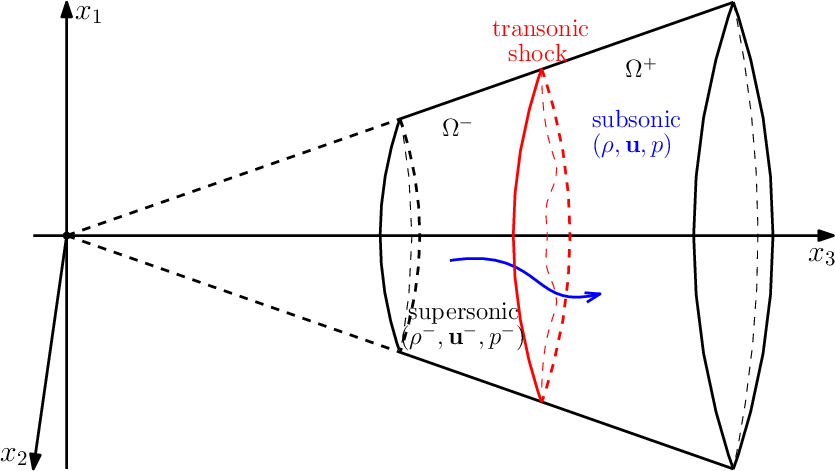}
\caption{Transonic shock in a three-dimensional divergent nozzle}\label{figure1}
\end{figure}
The transonic shock problems in divergent nozzles are motivated by the transonic shock problem in  \textit{a de Laval nozzle} (a convergent-divergent nozzle).
According to \cite[Chapter 146]{courant1999supersonic}, the flow pattern of the quasi-linear approximation in the diverging part of the de Laval nozzle varies depending only on the shape of the nozzle and the exit pressure.
The phenomenon has led many researchers to study the existence and stability of transonic shocks in a variety of related situations.
One of them is about the difference between flat nozzles and divergent nozzles.   The stability problem for transonic shocks in divergent nozzles when prescribing a suitable exit pressure is well-posed, while the problem in flat nozzles is ill-posed. For more information on this, one may refer to \cite{xin2009transonic, Yuan2008remark} and the references cited therein.
In this work, we establish the stability of three-dimensional weak solutions to \eqref{Euler-system-B} with  transonic shocks in divergent nozzles under small perturbations of the given upstream supersonic flows and the exit pressure.
For other studies on transonic shock problems, one may refer to \cite{bae2011transonic,  chen2003multidimensional, chen2007existence,  chen2018mathematics, Chen2022, FANG202162, Fang2021, park20213, weng2021structural} and the references cited therein.
In particular, one can refer to \cite{Chen2022, FANG202162, Fang2021} and the references cited therein for recent developments in the analysis of multidimensional transonic shocks.

Inspired by Bae and Feldman's work \cite{bae2011transonic} on transonic shocks for multi-dimensional non-isentropic potential flows, we study the existence and stability of transonic shocks for three-dimensional axisymmetric flows with nonzero vorticity and nonzero angular momentum density.
Recently, in \cite{park20213, weng2021structural},  the stability problem  solved by using the method of stream function formulation. 
In particular, in \cite{park20213}, the author studied the flows having $C^{1,\alpha}$ interior and $C^{\alpha}$ up to boundary regularity in some weighted H\"older normed space due to the corner singularity issues.
We use the same weighted H\"older norm to deal with the singularity issues, but not only get an improvement for the exponent $\alpha$ of the H\"older condition, but also use a new method other than a stream function formulation called the Helmholtz decomposition method.
That is the main new feature of this paper.

According to the fundamental theorem of vector calculus, if the velocity field ${\bf u}$ is continuously differentiable in a bounded domain, then it can be decomposed by the Helmholtz decomposition into a sum of an irrotational vector field and a solenoidal vector field in the form of
\begin{equation*}
{\bf u}=\nabla\varphi+\nabla\times{\bf V}.
\end{equation*} 
The recently published papers (cf. \cite{Bae:2021aa, bae2014subsonic,  bae2019contact, bae2019contact3D, bae20183, park2020transonic}) show that this decomposition is very useful to study the steady Euler/Euler-Poisson system.
The Helmholtz decomposition method is significant in that it allows us to investigate flows with a nonzero vorticity and build theories on them based on the study of irrotational flows.
We use this to study the stability of transonic shocks for the flows with nonzero vorticity and nonzero angular momentum density in divergent nozzles under perturbations of the exit pressure and the supersonic solution in the upstream region.
We first derive a free boundary problem with the newly introduced formulation of the Euler system for three-dimensional axisymmetric flows in divergent nozzles via the method of Helmholtz decomposition. 
We then construct an iteration scheme and use the Schauder fixed point theorem and weak implicit function theorem to solve the problem. 
To address technical problems that do not appear  in studies of potential flows, including the solvability of an elliptic equation with singular coefficients, we employ the method of \cite{bae20183,park2020transonic} for flows with nonzero vorticity and nonzero angular momentum density in three-dimensional cylinders.
To the best of our knowledge, this is the first study on transonic shock flows with nonzero vorticity and nonzero angular momentum density in divergent nozzles via the Helmholtz decomposition method.  

This paper is organized as follows.
In Section \ref{Sec-Main}, the problem and theorem are stated as Problem \ref{re-div-Prob} and Theorem \ref{Main-Theorem}, respectively.
In Section \ref{Sec-HD}, the problem is reformulated via the  Helmholtz decomposition method, and its solvability is stated as Theorem \ref{Helmholtz-Theorem}.
Then, we give an overview of the proof of Theorem \ref{Helmholtz-Theorem}. 
In the last part of Section \ref{Sec-HD}, we give an overview of the proof of Proposition \ref{free-v-Pro}, which is an essential step in the proof of Theorem \ref{Helmholtz-Theorem}. 
In Sections \ref{Sec-Pro}, we prove two lemmas required to prove Proposition \ref{free-v-Pro}.
In Section \ref{Sec-Mthm}, we prove Proposition \ref{free-v-Pro} and Theorem \ref{Helmholtz-Theorem} completely.

\section{Problem and Theorem}\label{Sec-Main}


\subsection{Background solution}\label{S-back}
We use a spherical coordinate system to study axisymmetric solutions.
Let $(r,\phi,\theta)$ be the spherical coordinates of ${\bf x}=(x_1,x_2,x_3)\in\mathbb{R}^3$, that is, 
\begin{equation*}
(x_1,x_2,x_3)=(r\sin\phi\cos\theta,r\sin\phi\sin\theta,r\cos\phi),\quad r\ge0,\,\,\phi\in[0,\pi],\,\,\theta\in\mathbb{T},
\end{equation*}
where $\mathbb{T}$ is a one dimensional torus with period $2\pi$. 
With this coordinate system, let us define a divergent nozzle $\Om$ by
\begin{equation*}
\Om:=\{{\bf x}\in\mathbb{R}^3:\, r_{\rm en}<r<r_{\rm ex},\,0\le\phi<\phi_0\}
\end{equation*}
for positive constants $r_{\rm en}$, $r_{\rm ex}\in\mathbb{R}^+$, and $\phi_0\in(0,\frac{\pi}{2})$.
We denote the entrance, wall, and the exit of $\Om$ as follows:
\begin{equation*}
\begin{split}
&\Gamma_{\rm en}:=\partial\Om\cap\{r=r_{\rm en},\,0\le\phi<\phi_0\},\quad\Gamma_{\rm w}:=\partial\Om\cap\{\phi=\phi_0\},\\
&\Gamma_{\rm ex}:=\partial\Om\cap\{r=r_{\rm ex},\,0\le\phi<\phi_0\}.
\end{split}
\end{equation*}
In this paper, we assume that $r_{\rm sh}>1$ and $r_{\rm ex}-r_{\rm sh}>1$ for simplification.

The existence of a potential radial solution to the Euler system with a transonic shock in $\Om$ is well known (cf. \cite[Section 2.4]{bae2011transonic}).
More precisely, we have a radial solution $(\rho,{\bf u},p)$ with a transonic shock $\Gamma_{\rm sh}$ satisfying the following properties:
\begin{itemize}
\item[(i)] (Shock)  For some positive constant $r_{\rm sh}$ with $r_{\rm en}\le r_{\rm sh}\le r_{\rm ex}$, $$\Gamma_{\rm sh}=\overline{\Om}\cap\{r=r_{\rm sh}\};$$
\item[(ii)] (Radial solutions) For $\Om^{\pm}:=\Om\cap\{\pm(r_{\rm sh}-r)<0\}$, $(\rho,{\bf u}, p)\in \left[C^0(\overline{\Om^{\pm}})\cap C^1(\Om^{\pm})\right]^5$ satisfies the Euler system \eqref{Euler-system-B} pointwise in $\Om^{\pm}$. For some radial functions $\rho^{\pm}=\rho^{\pm}(r;r_{\rm sh})$, $\varphi^\pm=\varphi^\pm(r;r_{\rm sh})$, and $p^{\pm}=p^{\pm}(r;r_{\rm sh})$ defined in $\Om$,
\begin{equation*}
\rho=\rho^{\pm},\quad{\bf u}=\nabla\varphi^{\pm},\quad p=p^{\pm}\quad\mbox{in}\quad\Om^{\pm}
\end{equation*}
and $(\rho^+,\nabla\varphi^+,p^+)(r_{\rm sh};r_{\rm sh})\ne(\rho^-,\nabla\varphi^-,p^-)(r_{\rm sh};r_{\rm sh})$, but $\varphi_0^+(r_{\rm sh};r_{\rm sh})=\varphi_0^-(r_{\rm sh};r_{\rm sh})$.
Furthermore, $\rho$, ${\bf u}$ and $p$ are positive functions and there exist positive constants $\mu_0^{\pm}(r_{\rm sh})$ such that $\rho^{\pm}(r;r_{\rm sh})\ge \mu_0^{\pm}(r_{\rm sh})>0$;
\item[(iii)] (Transonic) $\lvert\nabla\varphi^-\rvert>\sqrt{\frac{\gamma p^-}{\rho^-}}$ (supersonic) and $\lvert \nabla\varphi^+\rvert<\sqrt{\frac{\gamma p^+}{\rho^+}}$ (subsonic) hold;
\item[(iv)] (Entrance data) Given positive constants $(\rho_{\rm en},u_{\rm en},p_{\rm en})\in\mathbb{R}^3$ with $u_{\rm en}>\sqrt{\frac{\gamma p_{\rm en}}{\rho_{\rm en}}}$, it holds that 
\begin{equation*}
(\rho^-,\partial_r\varphi^-,p^-)=(\rho_{\rm en},u_{\rm en},p_{\rm en})\quad\mbox{on}\quad\Gamma_{\rm en};
\end{equation*}
\item[(v)] (Slip boundary condition) The solution satisfy the slip boundary condition on the wall, that is, $\nabla\varphi^{\pm}\cdot{\bf n}=0$ on $\Gamma_{\rm w}$ for the outward unit normal vector field ${\bf n}$;
\item[(vi)] (Rankine-Hugoniot conditions) On the transonic shock $\Gamma_{\rm sh}$,
\begin{equation}\label{K0-def}
\partial_r\varphi^+=\frac{K_0}{\partial_r\varphi^-}\mbox{ and }p^+=\rho^-\lvert\partial_r\varphi^-\rvert^2+p^--\rho^-K_0\,\,\mbox{for}\,\,K_0:=\frac{2(\gamma-1)B_0}{\gamma+1};
\end{equation}
\item[(vii)] (Bernoulli invariant) $B(\rho,{\bf u},p)\equiv B_0$ for a fixed constant $B_0>0$, that is,
\begin{equation}\label{Back-B}
\frac{1}{2}\lvert\partial_r\varphi^{\pm}\rvert^2+\frac{\gamma p^{\pm}}{(\gamma-1)\rho^{\pm}}=B_0\quad\mbox{in}\quad\overline{\Om};
\end{equation}
\item[(viii)] (Monotonicity) For a fixed $r_{\rm sh}$, the functions $\rho^{\pm}, p^{\pm},$ and $\partial_r\varphi^{\pm}$ have monotonic properties;
	\begin{itemize}
	\item[$\cdot$] $\rho^-$ and $p^-$ monotonically decrease while $\partial_r\varphi^-$ monotonically increases by $r$;
	\item[$\cdot$] $\rho^+$ and $p^+$ monotonically increase while $\partial_r\varphi^+$ monotonically decreases by $r$.
	\end{itemize}
\item[(ix)] (Admissibility) $\nabla\varphi^-\cdot{\bf n}>\nabla\varphi^+\cdot{\bf n}>0$ for the unit normal vector field ${\bf n}$ on $\Gamma_{\rm sh}$ pointing toward $\Om^+$.
\end{itemize}

\begin{definition} For a fixed $r_{\rm sh}$, we denote the solution satisfying {\rm (i)-(ix)} as 
\begin{equation*}
(\rho_0,\nabla\varphi_0,p_0)(r;r_{\rm sh}):=\left\{\begin{split}(\rho_0^-,\nabla\varphi_0^-,p_0^-)(r;r_{\rm sh})\mbox{ if }r<r_{\rm sh}\\
										(\rho_0^+,\nabla\varphi_0^+,p_0^+)(r;r_{\rm sh})\mbox{ if }r>r_{\rm sh}\end{split}\right.,
\end{equation*}
and call it the background solution with the transonic shock on $\Gamma_{\rm sh}={\Om}\cap\{r=r_{\rm sh}\}$.
\end{definition}

In \cite{Yuan2008remark}, it was proved that the monotonicity of the exit values of $(\rho_0,\partial_r\varphi_0,p_0)$ depends on the location of shocks $r_{\rm sh}$.
And the following proposition was verified.

\begin{proposition}\cite[Theorem 2.1]{Yuan2008remark}
Set $p_{\min}:=p^+_0(r_{\rm ex};r_{\rm ex})$ and $p_{\max}:=p_0^+(r_{\rm ex};r_{\rm en})$. Then $p_{\min}<p_{\max}$ holds. 
For any $p_c\in(p_{\min},p_{\max})$, there exists a unique $r_c\in(r_{\rm en},r_{\rm ex})$ such that the background solution $(\rho_0,\nabla\varphi_0,p_0)(r;r_c)$ satisfies $p_0^+(r_{\rm ex},r_c)=p_c$.  
\end{proposition}



\subsection{Main results}
The goal of this work is to prove the unique existence of an axisymmetric solution to the full Euler system with a transonic shock in the sense of Definition \ref{Def-shock} when we perturb  sufficiently small the background supersonic solution in the upstream region and the background pressure on the exit.

\begin{definition} 
Let us define an orthonormal basis $\{{\bf e}_r,{\bf e}_{\phi},{\bf e}_\theta\}$ for the spherical coordinates as follows:
\begin{equation*}
\begin{split}
&{\bf e}_r:=(\sin\phi\cos\theta,\sin\phi\sin\theta,\cos\phi),\, {\bf e}_\phi:=(\cos\phi\cos\theta,\cos\phi\sin\theta,-\sin\phi),\\
&{\bf e}_{\theta}:=(-\sin\theta,\cos\theta, 0). 
\end{split}
\end{equation*}
Let $D$ be an open subset of $\mathbb{R}^3$ and let ${\bf x}=(x_1,x_2,x_3)\in D$.
\begin{itemize}
\item[(i)] A function $f:D\to\mathbb{R}$ is {\emph{axisymmetric}} if $f({\bf x})=\tilde{f}(r,\phi)$ for some $\tilde{f}:\mathbb{R}^2\to\mathbb{R}$ so its value is independent of $\theta$.
\item[(ii)] A vector-valued function ${\bf F}:D\to\mathbb{R}^3$ is {\emph{axisymmetric}} if ${\bf F}({\bf x})=F_r(r,\phi){\bf e}_r+F_\phi(r,\phi){\bf e}_\phi+F_{\theta}(r,\phi){\bf e}_\theta$ for some axisymmetric functions $F_r$, $F_\phi$, and $F_\theta$.
\end{itemize}
\end{definition}

To describe our problem and theorem, we introduce the weighted H\"older norm defined as follows:
Let $\Omega\subset\mathbb{R}^n$ be an open bounded connected set, and let $\Gamma$ be a closed portion of $\partial\Omega$. 
For ${\bf x}, {\bf y}\in\Omega$, set $\delta_{\bf x}$ and $\delta_{{\bf x,y}}$ as 
\begin{equation*}
\delta_{\bf x}:=\inf_{{\bf z}\in\Gamma}\lvert{\bf x}-{\bf z}\rvert\quad\mbox{and}\quad\delta_{{\bf x,y}}:=\min(\delta_{\bf x},\delta_{\bf y}).
\end{equation*}
For $\alpha\in(0,1)$, $k\in\mathbb{R}$ and $m\in\mathbb{Z}^+$, define  the weighted H\"older norms by 
\begin{eqnarray*}
\begin{split}
&\|u\|_{m,0,\Omega}^{(k,\Gamma)}:=\sum_{0\le\lvert \beta\rvert\le m}\sup_{{\bf x}\in\Omega}\delta_{\bf x}^{\max(\lvert \beta\rvert+k,0)}\lvert D^{\beta}u({\bf x})\rvert,\\
&[u]_{m,\alpha,\Omega}^{(k,\Gamma)}:=\sup_{\lvert \beta\rvert=m}\sup_{\substack{{\bf x,y}\in\Omega,\\{\bf x}\ne{\bf y}}}\delta_{\bf x,y}^{\max(m+\alpha+k,0)}\frac{\lvert D^{\beta}u({\bf x})-D^{\beta}u({\bf y})\rvert}{\lvert {\bf x}-{\bf y}\rvert^{\alpha}},\\
&\|u\|_{m,\alpha,\Omega}^{(k,\Gamma)}:=\|u\|_{m,0,\Omega}^{(k,\Gamma)}+[u]_{m,\alpha,\Omega}^{(k,\Gamma)},
\end{split}
\end{eqnarray*}
where we set $D^{\beta}:=\partial_{x_1}^{\beta_1}\dots\partial_{x_n}^{\beta_n}$ for a multi-index $\beta=(\beta_1,\ldots,\beta_n)$ with $\beta_j\in\mathbb{Z}^+$ and $\lvert \beta\rvert=\sum_{j=1}^n\beta_j$. 
We denote by $C^{m,\alpha}_{(k,\Gamma)}(\Omega)$ the completion of the set of all smooth functions whose $\|\cdot\|_{m,\alpha,\Omega}^{(k,\Gamma)}$ norms are finite.

Our concern is to solve the following problem. 
\begin{problem}\label{S-div-Prob} 
Let $(\rho^-,{\bf u}^-,p^-)\in [C^{2,\alpha}(\overline{\Om})]^5$ be an axisymmetric supersonic solution to the Euler system \eqref{Euler-system-B} in $\Om$ with $B=B_0$, and suppose that ${\bf u}^-\cdot{\bf e}_{\phi}=0$ holds on $\Gamma_{\rm w}$.
For any fixed $\alpha\in(\frac{1}{2},1)$ and a given radial function $p_{\rm ex}$ defined on $\Gamma_{\rm ex}$, 
assume that
{\small \begin{equation}\label{3D-rem11}
\begin{split}
\sigma(\rho^-,{\bf u}^-,p^-,p_{\rm ex})&:=\|(\rho^-,{\bf u}^-,p^-)-(\rho_0^-,\nabla\varphi_0^-,p_0^-)\|_{2,\alpha,\Om}+\|p_{\rm ex}-p_0^+\|^{(-\alpha,\partial\Gamma_{\rm ex})}_{1,\alpha,\Gamma_{\rm ex}}\\
&\le \sigma_0
\end{split}
\end{equation}}
with sufficiently small $\sigma_0>0$ to be specified later.

Find a weak solution $U=(\rho,{\bf u}, p)$ to the Euler system \eqref{Euler-system-B} with a transonic shock
$$\mathfrak{S}_f:={\Om}\cap\{r=f(\phi)\}$$
in the sense of Definition \ref{Def-shock} in $\Om$  such that the following properties hold:
\begin{itemize}
\item[(i)] 
$\rho>0$ holds in $\overline{\Om}$;
\item[(ii)] In the upstream region $\Om_f^-:=\Om\cap\{r<f(\phi)\}$,
$U=(\rho^-,{\bf u}^-,p^-)$ holds;
\item[(iii)] In the downstream region $\Om_f^+:=\Om\cap\{r>f(\phi)\}$,
 ${\lvert{\bf u}\rvert}<c(\rho,p)$ for the sound speed $c(\rho,p):=\sqrt{\frac{\gamma p}{\rho}};$
 \item[(iv)] On the exit $\Gamma_{\rm ex}$, 
 $p=p_{\rm ex}$ holds;
 \item[(v)] On the wall $\Gamma_{\rm w}$,
$U$ satisfies the slip boundary condition, that is,
${\bf u}\cdot{\bf e}_\phi=0;$
\item[(vi)] On the transonic shock $\mathfrak{S}_f$,
$U$ satisfies the Rankine-Hugoniot conditions
\begin{equation*}\label{3D-BC-Cont1}
[{\bf u}\cdot{\bm\tau}_f]_{\mathfrak{S}_f}=[\rho {\bf u}\cdot{\bf n}_f]_{\mathfrak{S}_f}=[\rho({\bf u}\cdot{\bf n}_f)^2+p]_{\mathfrak{S}_f}=0,
\end{equation*}
where ${\bm \tau}_f$ and ${\bf n}_f$ denote a unit tangent vector field and unit normal vector field on $\mathfrak{S}_f$, respectively;
\item[(vii)] For the given background constant $B_0$ in \eqref{Back-B}, $B(\rho,{\bf u},p)\equiv B_0$ holds in $\overline{\Om}$.
\end{itemize}
\end{problem}

\begin{figure}[!h]
\centering
\includegraphics[scale=0.7]{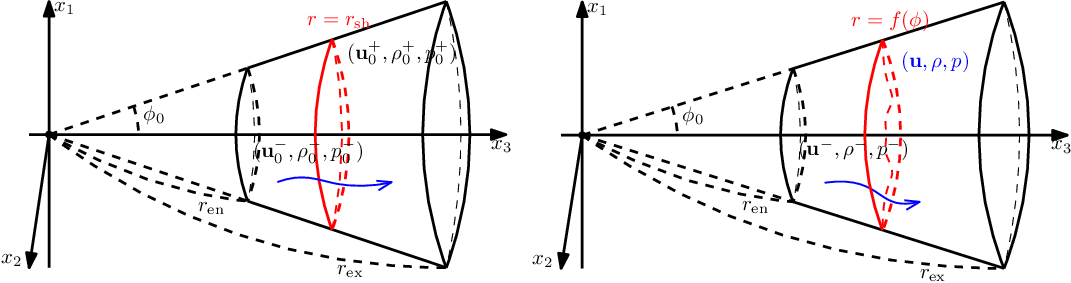}
\caption{Left:background solution, right:problem}
\end{figure}

We can rewrite Problem \ref{S-div-Prob} as a free boundary problem.
\begin{problem}\label{re-div-Prob}
Under the same assumptions of Problem \ref{S-div-Prob}, find a function
$$f:[0,\phi_0]\longrightarrow\left(r_{\rm sh}-\frac{1}{4},r_{\rm sh}+\frac{1}{4}\right)$$ and a weak solution $U=(\rho,{\bf u},p)$ to the Euler system \eqref{Euler-system-B} in $\Om_f^+:=\Om\cap\{r>f(\phi)\}$ such that the following properties hold:
\begin{itemize}
\item[(i)]
$\rho>0$ holds in $\overline{\Om_f^+}$;
\item[(ii)] In  $\Om_f^+$,
 ${\lvert{\bf u}\rvert}<c(\rho,p)$ for the sound speed $c(\rho,p):=\sqrt{\frac{\gamma p}{\rho}};$
 \item[(iii)] On the exit $\Gamma_{\rm ex}$, 
 $p=p_{\rm ex}$ holds;
\item[(iv)] On $\Gamma_{{\rm w},f}^+(:=\partial\Om_f^+\cap\Gamma_{\rm w})$, 
${\bf u}\cdot{\bf e}_\phi=0$ holds;
\item[(v)] On $\mathfrak{S}_f$, $U$ satisfies the boundary conditions
\begin{equation}\label{RH-re}
\left\{\begin{split}
&{\bf u}\cdot{\bm \tau}_f={\bf u}^-\cdot{\bm\tau}_f,\quad \rho{\bf u}\cdot{\bf n}_f=\rho^-{\bf u}^-\cdot{\bf n}_f,\\
&\rho({\bf u}\cdot{\bf n}_f)^2+p=\rho^-({\bf u}^-\cdot{\bf n}_f)^2+p^-,
\end{split}\right.
\end{equation}
where ${\bm \tau}_f$ and ${\bf n}_f$ denote a unit tangent vector field and unit normal vector field on $\mathfrak{S}_f$, respectively;
\item[(vi)] For the given constant $B_0$ in \eqref{Back-B}, $B(\rho,{\bf u},p)\equiv B_0$ holds in $\overline{\Om_f^+}$.
\end{itemize}
\end{problem}

Problem \ref{re-div-Prob} has a unique solution if the upstream supersonic solution and the exit pressure are small perturbation of the background solution. 

\begin{theorem}\label{Main-Theorem} 
Let $(\rho^-,{\bf u}^-,p^-,p_{\rm ex})$ be from Problem \ref{S-div-Prob}.

${\rm (a)}$ For any fixed $\alpha\in(\frac{1}{2},1)$, there exists a small constant $\sigma_1>0$ depending only on $\rho_0^{\pm},\partial_r\varphi_0^{\pm},p_0^{\pm},\gamma,\alpha,r_{\rm en},r_{\rm ex}$, and $\phi_0$ so that if
\begin{equation*}
\sigma(\rho^-,{\bf u}^-,p^-,p_{\rm ex})
\le \sigma_1,
\end{equation*}
 then there exists an axisymmetric solution $U=(\rho,{\bf u},p)$ of Problem \ref{re-div-Prob} with a transonic shock $\mathfrak{S}_f(={\Omega_f^+}\cap\{r=f(\phi)\})$ that satisfies 
{\small \begin{equation}\label{Thm2.1-uniq-est} 
\|f-r_{\rm sh}\|_{2,\alpha,(0,\phi_0)}^{(-1-\alpha,\{\phi=\phi_0\})}+\|(\rho,{\bf u},p)-(\rho_0^+,\nabla\varphi_0^+,p_0^+)\|_{1,\alpha,\Om_f^+}^{(-\alpha,\Gamma_{\rm w})}\le C\sigma(\rho^-,{\bf u}^-,p^-,p_{\rm ex}),
\end{equation}}
where the constant $C>0$ depends only on $(\rho_0^{\pm},\partial_r\varphi_0^{\pm},p_0^{\pm},\gamma,\alpha,r_{\rm en},r_{\rm ex},\phi_0)$.

${\rm (b)}$ There exists a small constant $\tilde{\sigma}_1\in(0,\sigma_1]$ depending only on  $\rho_0^{\pm},\partial_r\varphi_0^{\pm},p_0^{\pm},\gamma,\alpha,$ $r_{\rm en},$ $r_{\rm ex}$, and $\phi_0$ so that if
$$\sigma(\rho^-,{\bf u}^-,p^-,p_{\rm ex})\le \tilde{\sigma}_1,$$
then the axisymmetric solution that satisfies the estimate \eqref{Thm2.1-uniq-est} is unique.
\end{theorem}


\section{Reformulation of Theorem \ref{Main-Theorem} via Helmholtz decomposition}\label{Sec-HD}
\subsection{Main Theorem}
Suppose that the smooth solution $(\rho, {\bf u},p)$ of the Euler system \eqref{Euler-system-B} is axisymmetric, that is,
\begin{equation*}
\rho=\rho(r,\phi),\quad {\bf u}=u_r(r,\phi){\bf e}_r+u_{\phi}(r,\phi){\bf e}_{\phi}+u_{\theta}(r,\phi){\bf e}_{\theta},\quad p=p(r,\phi),
\end{equation*}
and suppose that $\rho>0$, $u_r>0$, $p>0$, and the Bernoulli invariant $B$ is a constant $B_0$.
For $r>0$ and $\phi\in(0,\pi)$, the Euler system is equivalent to the following system:
\begin{equation}\label{axi-sys}
\left\{\begin{split}
&\partial_r(r^2\sin\phi \rho u_r)+\partial_{\phi}(r\sin\phi\rho u_{\phi})=0,\\
&\rho u_r\partial_r u_{\phi}+\frac{1}{r}\rho u_{\phi}\partial_{\phi}u_{\phi}+\frac{1}{r}\partial_{\phi}(S\rho^{\gamma})+\frac{\rho u_r u_{\phi}}{r}-\frac{\rho \Lambda^2\cot\phi}{r^3\sin^2\phi}=0,\\
&(r^2\sin\phi \rho u_r)\partial_r\Lambda+(r\sin\phi\rho u_{\phi})\partial_{\phi}\Lambda=0,\\
&(r^2\sin\phi \rho u_r)\partial_rS+(r\sin\phi\rho u_{\phi})\partial_{\phi}S=0,\\
\end{split}\right.
\end{equation}
where $\Lambda$ and $S$ are defined by 
\begin{equation*}
\Lambda(r,\phi):=(r\sin\phi)u_{\theta}(r,\phi)\quad\mbox{and}\quad S(r,\phi):=\left(\frac{p}{\rho^{\gamma}}\right)(r,\phi).
\end{equation*}
We call $\Lambda$ and $S$ respectively the angular momentum density and entropy of the flow.

For a function $f:[0,\phi_0]\to(r_{\rm sh}-\frac{1}{4},r_{\rm sh}+\frac{1}{4})$ to be determined,
we express ${\bf u}$ as
\begin{equation*}
{\bf u}=\nabla\varphi+\mbox{curl}{\bf V}\quad\mbox{in}\quad\Om_f^+
\end{equation*}
for axisymmetric functions $\varphi=\varphi(r,\phi)$ and ${\bf V}=h(r,\phi){\bf e}_{\phi}+\psi(r,\phi){\bf e}_{\theta}$.
If $\varphi$, $h$ and $\psi$ are $C^2$, then a direct computation yields
\begin{equation}\label{u-def-q}
{\bf u}=\nabla\varphi+\nabla\times(\psi{\bf e}_{\theta})+\left(\frac{\Lambda}{r\sin\phi}\right){\bf e}_{\theta}
=:{\bf q}\left(\nabla\varphi,\nabla\times(\psi{\bf e}_{\theta}),\frac{\Lambda}{r\sin\phi}\right)
\end{equation}
and
\begin{equation}\label{u-com}
u_r=\partial_r\varphi+\frac{1}{r\sin\phi}\partial_{\phi}((\sin\phi) \psi),\,\, u_{\phi}=\frac{1}{r}(\partial_{\phi}\varphi)-\frac{1}{r}\partial_r(r\psi),\,\, u_{\theta}=\frac{\Lambda}{r\sin\phi}.
\end{equation}
With this expression, the density $\rho$ can be written in terms of $(S,\nabla\varphi,\nabla\times(\psi{\bf e}_{\theta}),\frac{\Lambda}{r\sin\phi}{\bf e}_{\theta})$ as follows:
\begin{equation*}
\rho=\varrho\left(S,\nabla\varphi+\nabla\times(\psi{\bf e}_{\theta})+\frac{\Lambda}{r\sin\phi}{\bf e}_{\theta}\right)
\end{equation*}
for a function $\varrho$ defined by 
\begin{equation}\label{def-H}
\varrho(\eta,{\bf q}):=\left[\frac{\gamma-1}{\gamma\eta}\left(B_0-\frac{1}{2}\lvert{\bf q}\rvert^2\right)\right]^{1/(\gamma-1)}
\mbox{ for }\eta\in\mathbb{R}, \,{\bf q}\in\mathbb{R}^3.
\end{equation}
For notational simplicity, let us set 
\begin{equation*}
{\bf t}:={\bf q}\left(\nabla\varphi,\nabla\times(\psi{\bf e}_{\theta}),\frac{\Lambda}{r\sin\phi}\right)-\nabla\varphi\left(=\mbox{curl}{\bf V}\right).
\end{equation*}

When the entropy is strictly positive, i.e., $S>0$,  the Euler system is equivalent to the following system: 
\begin{equation}\label{E-re}
\left\{\begin{split}
&\mbox{div}\left({\bf A}(\nabla\varphi,{\bf t})\right)=0,\\
&-\Delta(\psi{\bf e}_{\theta})=G(r,\phi,S,\Lambda,\partial_{\phi}S,\partial_{\phi}\Lambda,\nabla\varphi+{\bf t}){\bf e}_{\theta},\\
&{\bf A}(\nabla\varphi,{\bf t})\cdot\nabla \Lambda= 0,\\
&{\bf A}(\nabla\varphi,{\bf t})\cdot\nabla S=0,
\end{split}\right.
\end{equation}
where ${\bf A}$ and $G$ are defined by
\begin{equation}\label{def-H-G}
\left.\begin{split}
&{\bf A}({\bf s}_1,{\bf s}_2):=\left(B_0-\frac{1}{2}\lvert{\bf s}_1+{\bf s}_2\rvert^2\right)^{\frac{1}{\gamma-1}}({\bf s}_1+{\bf s}_2),\\
&G(r,\phi,\eta_1,\eta_2,\eta_3,\eta_4,{\bf q}):=\frac{1}{r({\bf q}\cdot{\bf e}_r)}\left(\frac{\eta_3\varrho^{\gamma-1}(\eta_1, {\bf q})}{\gamma-1}+\frac{\eta_2\eta_4}{r^2\sin^2\phi}\right)
\end{split}
\right.
\end{equation}
for ${\bf s}_1,{\bf s}_2,{\bf q}\in\mathbb{R}^3$, $\eta_1,\eta_2,\eta_3,\eta_4\in\mathbb{R}$.

The given axisymmetric supersonic solution $(\rho^-,{\bf u}^-,p^-)$ can be written  as 
\begin{equation}\label{super-HD}
\rho^-=\varrho(S^-,{\bf u}^-),\,\, {\bf u}^-= {\bf q}\left(\nabla\varphi^-,\nabla\times(\psi^-{\bf e}_{\theta}),\frac{\Lambda^-}{r\sin\phi}\right),\,\, p^-=S^-\varrho^{\gamma}(S^-,{\bf u}^-)
\end{equation}
for axisymmetric functions $S^-=S^-(r,\phi)$, $\varphi^-=\varphi^-(r,\phi)$, $\psi^-=\psi^-(r,\phi)$, $\Lambda^-=\Lambda^-(r,\phi)$.
For notational simplicity, let us set ${\bf t}^-$ as
\begin{equation}\label{t--}
{\bf t}^-:={\bf q}\left(\nabla\varphi^-,\nabla\times(\psi^-{\bf e}_{\theta}),\frac{\Lambda^-}{r\sin\phi}\right)-\nabla\varphi^-.
\end{equation}



Now we derive the boundary conditions for $(f,S,\Lambda,\varphi,\psi{\bf e}_{\theta})$ to satisfy the physical boundary conditions in Problem \ref{re-div-Prob}:
\begin{itemize}
\item[(a)] (Slip boundary condition) By  \eqref{u-com}, the slip boundary condition (${\bf u}\cdot{\bf e}_{\phi}=0$ on $\Gamma_{\rm w}$) is equivalent to 
	\begin{equation*}
	\frac{1}{r}\left(\partial_{\phi}\varphi-\partial_r(r\psi)\right)=0\,\,\mbox{on}\,\,\Gamma_{\rm w}.
	\end{equation*}
	If we prescribe the boundary conditions for $(\varphi,\psi{\bf e}_{\theta})$ on $\Gamma_{\rm w}$ as 
	\begin{equation}\label{bd-W}
	\nabla\varphi\cdot{\bf e}_{\phi}=0\mbox{ and }\psi{\bf e}_{\theta}={\bf 0}\mbox{ on }\Gamma_{\rm w},
	\end{equation}
	then the slip boundary condition holds.
\item[(b)] (Exit pressure condition) In terms of $(S,\Lambda,\varphi,\psi{\bf e}_{\theta})$, the exit condition can be rewritten as 
\begin{equation}\label{bd-EX}
S\varrho^{\gamma}\left(S,{\bf q}\left(\nabla\varphi,\nabla\times(\psi{\bf e}_{\theta}),\frac{\Lambda}{r\sin\phi}\right)\right)=p_{\rm ex}\mbox{ on }\Gamma_{\rm ex}.
\end{equation}
\item[(c)] (Rankine-Hugoniot conditions) In terms of $(f,S,\Lambda, \varphi,\psi{\bf e}_{\theta})$,  the Rankine-Hugoniot conditions in \eqref{RH-re} can be rewritten as 
\begin{eqnarray}
&\label{HD-RH1}&{\bf q}\cdot{\bm\tau}_{f}={\bf u}^-\cdot{\bm\tau}_f,\quad{\bf q}\cdot{\bf e}_{\theta}={\bf u}^-\cdot{\bf e}_{\theta},\\
&\label{HD-RH2}&\varrho(S,{\bf q})({\bf q}\cdot{\bf n}_f)=\rho^-({\bf u}^-\cdot{\bf n}_f),\\
&\label{HD-RH3}&\varrho(S,{\bf q})({\bf q}\cdot{\bf n}_f)^2+S\varrho^{\gamma}(S,{\bf q})=\rho^-({\bf u}^-\cdot{\bf n}_f)^2+p^-
\end{eqnarray}
for ${\bf q}$, ${\bm \tau}_f$, and ${\bf n}_f$ defined by 
\begin{equation*}
\begin{split}
&{\bf q}:={\bf q}\left(\nabla\varphi,\nabla\times(\psi{\bf e}_{\theta}),\frac{\Lambda}{r\sin\phi}\right),\\
&{\bm \tau}_{f}:=\frac{f'(\phi){\bf e}_r+\frac{1}{f(\phi)}{\bf e}_{\phi}}{\sqrt{\frac{1}{|f(\phi)|^2}+\lvert f'(\phi)\rvert^2}},\quad\mbox{and}\quad
{\bf n}_{f}:=\frac{\frac{1}{f(\phi)}{\bf e}_r-f'(\phi){\bf e}_{\phi}}{\sqrt{\frac{1}{|f(\phi)|^2}+\lvert f'(\phi)\rvert^2}}.
\end{split}
\end{equation*}
\item[(i)]
Obviously, the conditions in \eqref{HD-RH1} are equivalent to
\begin{equation*}
(\nabla\varphi+{\bf t})\cdot{\bm\tau}_{f}=(\nabla\varphi^-+{\bf t}^-)\cdot{\bm\tau}_{f},
\quad\frac{\Lambda}{r\sin\phi}=\frac{\Lambda^-}{r\sin\phi}.
\end{equation*}
If we prescribe the boundary conditions for $(\varphi,\psi{\bf e}_{\theta},\Lambda)$ on $\mathfrak{S}_f$ as
\begin{equation}\label{S-phi-Lam}
\varphi=\varphi^-,\quad [\nabla\times(\psi{\bf e}_{\theta})]\cdot{\bm\tau}_{f}=[\nabla\times(\psi^-{\bf e}_{\theta})]\cdot{\bm\tau}_{f},\quad \Lambda=\Lambda^-,
\end{equation}
then the conditions in \eqref{HD-RH1} hold.
The condition for $\psi{\bf e}_{\theta}$ in \eqref{S-phi-Lam} is equivalent to 
\begin{equation*}\label{S-psi}
-\nabla(\psi{\bf e}_{\theta})\cdot{\bf n}_f+\mu(f,f')(\psi{\bf e}_{\theta})=\mathcal{A}(f,f'){\bf e}_{\theta}
\end{equation*}
for $\mu(f,f')$ and $\mathcal{A}(f,f')$ defined by 
\begin{equation}\label{def-mu}
{\small
\begin{split}
\mu(f,f'):=&\frac{f'(\phi)\cos\phi}{f(\phi)\sin\phi}-\frac{1}{|f(\phi)|^2},\\
\mathcal{A}(f,f'):=&\frac{1}{\sqrt{\lvert f'(\phi)\rvert^2+\frac{1}{|f(\phi)|^2}}}\left[\frac{f'(\phi)}{f(\phi)\sin\phi}\partial_{\phi}((\sin\phi)\psi^-)-\frac{\psi^-+f\partial_r\psi^-}{|f(\phi)|^2}\right].
\end{split}}
\end{equation}
\item[(ii)] Note that 
\begin{equation}\label{B-equal}
B\equiv B_0=\frac{\lvert{\bf u}^{\pm}\rvert^2}{2}+\frac{\gamma p^{\pm}}{(\gamma-1)\rho^{\pm}}\quad\mbox{in}\quad\overline{\Om}.
\end{equation}
If $(\varphi,\psi{\bf e}_{\theta},\Lambda,S,f)$ satisfy
\begin{eqnarray}
&\label{K-mul}&-\nabla\varphi\cdot{\bf n}_f=-\frac{K_s(f')}{{\bf u}^-\cdot{\bf n}_f}+{\bf t}\cdot{\bf n}_f,\\
&\label{S-shock}&S=\left[\rho^-({\bf u}^-\cdot{\bf n}_f)^2+p^--\rho^-K_s(f')\right]\left[\frac{\rho^-({\bf u}^-\cdot{\bf n}_f)^2}{K_s(f')}\right]^{-\gamma}\mbox{ on }\mathfrak{S}_f
\end{eqnarray}
for $K_s(f')$ defined by 
\begin{equation}\label{def-KS}
K_s(f'):=\frac{2(\gamma-1)}{\gamma+1}\left(\frac{1}{2}\lvert {\bf u^-}\cdot{\bf n}_f\rvert^2+\frac{\gamma p^-}{(\gamma-1)\rho^-}\right),
\end{equation}
then a direct computation with using the definition \eqref{def-H} of $\varrho$  and \eqref{B-equal}-\eqref{S-shock} implies that 
\begin{equation}\label{re-H-RH}
\varrho\left(S,{\bf q}\right)=\frac{\rho^-({\bf u}^-\cdot{\bf n}_f)^2}{K_s(f')}\quad\mbox{on}\quad\mathfrak{S}_f.
\end{equation}
Combining \eqref{K-mul} and \eqref{re-H-RH} gives \eqref{HD-RH2}.
\item[(iii)] A direct computation with using \eqref{HD-RH2}, \eqref{S-shock}, and \eqref{re-H-RH} implies \eqref{HD-RH3}.
Then, we prescribe the boundary conditions for $(f,S,\Lambda, \varphi,\psi{\bf e}_{\theta})$ on $\mathfrak{S}_f$ as follows:
\begin{equation}\label{boundary-HD}
\left\{
\begin{split}
	\begin{split}
	&S=S_{\rm sh}(f'),\quad \Lambda=\Lambda^-,\quad \varphi=\varphi^-,\\
	&-\nabla\varphi\cdot{\bf n}_f=-\frac{K_s(f')}{{\bf u}^-\cdot{\bf n}_f}+{\bf t}\cdot{\bf n}_f,\\
	&-\nabla(\psi{\bf e}_{\theta})\cdot{\bf n}_f+\mu(f,f')(\psi{\bf e}_{\theta})=\mathcal{A}(f,f'){\bf e}_{\theta}	\end{split}
\end{split}\right.
\end{equation} 
with a function $S_{\rm sh}(f')$ defined by 
\begin{equation}\label{def-A}
\begin{split}
&S_{\rm sh}(f'):=\left[\rho^-({\bf u}^-\cdot{\bf n}_f)^2+p^--\rho^-K_s(f')\right]\left[\frac{\rho^-({\bf u}^-\cdot{\bf n}_f)^2}{K_s(f')}\right]^{-\gamma}.\\
\end{split}
\end{equation}
\end{itemize}

Now we state the main results in terms of $(f,S,\Lambda, \varphi,\psi{\bf e}_{\theta})$ as Theorem \ref{Helmholtz-Theorem} below. 
\begin{theorem}\label{Helmholtz-Theorem}
Let $(\rho^-,{\bf u}^-,p^-,p_{\rm ex})$ be from Problem \ref{S-div-Prob}.

${\rm (a)}$
For any $\alpha\in(\frac{1}{2},1)$, there exists a small constant $\sigma_2>0$ depending only on $\partial_r\varphi_0^{\pm},\rho_0^{\pm},p_0^\pm,\gamma,\alpha,r_{\rm en},r_{\rm ex}$, and $\phi_0$ so that if
\begin{equation}\label{enu-sigma}
\sigma(\rho^-,{\bf u}^-,p^-,p_{\rm ex})\le\sigma_2,
\end{equation}
then the free boundary problem \eqref{E-re} with boundary conditions \eqref{bd-W}-\eqref{bd-EX} and \eqref{boundary-HD} has an axisymmetric solution $(f, S, \Lambda, \varphi, \psi{\bf e}_{\theta})$ that satisfies
\begin{equation}\label{Thm-HD-est}
\begin{split}
&\|f-r_{\rm sh}\|_{2,\alpha,(0,\phi_0)}^{(-1-\alpha,\{\phi=\phi_0\})}
+\|(S,\frac{\Lambda}{r\sin\phi}{\bf e}_{\theta})-(S_0^+,{\bf 0})\|_{1,\alpha,\Om_f^+}^{(-\alpha,\Gamma_{\rm w})}\\
&+\|\varphi-\varphi_0^+\|_{2,\alpha,\Om_f^+}^{(-1-\alpha,\Gamma_{\rm w})}+\|\psi{\bf e}_{\theta}\|_{2,\alpha,\Om_f^+}^{(-1-\alpha,\Gamma_{\rm w})}\le C\sigma(\rho^-,{\bf u}^-,p^-,p_{\rm ex})
\end{split}
\end{equation}
for $S_0^+:=\frac{p_0^+}{(\rho_0^+)^{\gamma}}$ and a constant $C>0$ depending only on $\partial_r\varphi_0^{\pm},\rho_0^{\pm},p_0^\pm,\gamma,\alpha,r_{\rm en},r_{\rm ex}$, and $\phi_0$.

${\rm (b)}$ There exists a small constant $\tilde{\sigma}_2\in(0,\sigma_2]$ depending only on $\partial_r\varphi_0^{\pm},\rho_0^{\pm},p_0^\pm,\gamma,\alpha,$ $r_{\rm en}$, $r_{\rm ex}$, and $\phi_0$ so that if
\begin{equation*}
\sigma(\rho^-,{\bf u}^-,p^-,p_{\rm ex})\le\tilde{\sigma}_2,
\end{equation*}
then the solution that satisfies the estimate \eqref{Thm-HD-est} is unique. 
\end{theorem}
We can easily prove that Theorem \ref{Helmholtz-Theorem} implies Theorem \ref{Main-Theorem}. So the rest of the paper is devoted to proving Theorem \ref{Helmholtz-Theorem}.

Hereafter, a constant $C$ is said to be chosen depending only on the data if $C$ is chosen depending only on $\partial_r\varphi_0^{\pm},\rho_0^{\pm},p_0^{\pm},\gamma,\alpha,r_{\rm en},r_{\rm ex}$, and $\phi_0$.
Throughout this paper, each estimate constant $C$ is regarded to be depending only on the data unless otherwise specified.

\subsection{Outline of the proof of Theorem \ref{Helmholtz-Theorem}}\label{outline}

Let $v_{\rm ex}:\Gamma_{\rm ex}\to \mathbb{R}$ be an axisymmetric function and a small perturbation of the constant $v_c$ defined by 
\begin{equation}\label{v-c-def}
v_c:=\left(B_0-\frac{\lvert \partial_r\varphi_0^+(r_{\rm ex})\rvert^2}{2}\right)^{\frac{1}{\gamma-1}}\partial_r\varphi_0^+(r_{\rm ex}).
\end{equation}
We first solve the following free boundary problem for $(f,S,\Lambda,\varphi,{\bf W})$ associated with $v_{\rm ex}$:
\begin{equation}\label{free-v}
\begin{split}
&\left\{\begin{split}
	&\mbox{div}\left({\bf A}(\nabla\varphi,{\bf t})\right)=0\\
	&{\bf A}(\nabla\varphi,{\bf t})\cdot\nabla S=0,\quad
	{\bf A}(\nabla\varphi,{\bf t})\cdot\nabla \Lambda=0\\
	&-\Delta{\bf W}=G(r,\phi,S,\Lambda,\partial_{\phi}S,\partial_{\phi}\Lambda,\nabla\varphi+{\bf t}){\bf e}_{\theta}
	\end{split}\right.\quad\mbox{in}\quad\Om_f^+,\\
&\left\{\begin{split}
		&\varphi=\varphi^-,\quad-\nabla\varphi\cdot{\bf n}_f=-\frac{K_s(f')}{{\bf u}^-\cdot{\bf n}_f}+{\bf t}\cdot{\bf n}_f\quad\mbox{on}\quad\mathfrak{S}_f,\\
	&\nabla\varphi\cdot{\bf e}_{\phi}=0\quad\mbox{on}\quad\Gamma_{\rm w}\cap\partial\Om_f^+,\\
	&\left(B_0-\frac{\lvert\nabla\varphi+{\bf t}\rvert^2}{2}\right)^{\frac{1}{\gamma-1}}\nabla\varphi\cdot{\bf e}_r=v_{\rm ex}\quad\mbox{on}\quad\Gamma_{\rm ex},\\
	&S=S_{\rm sh}(f'),\quad \Lambda=\Lambda^-\quad\mbox{on}\quad\mathfrak{S}_f,\\
	&-\nabla{\bf W}\cdot{\bf n}_{f}+\mu(f,f'){\bf W}=\mathcal{A}(f,f'){\bf e}_{\theta}\,\mbox{ on }\,\mathfrak{S}_f,\quad
	{\bf W}={\bf 0}\,\mbox{ on }\,\partial\Om_f^+\backslash\mathfrak{S}_f,
	\end{split}\right.
\end{split}
\end{equation}
for ${\bf A}$, ${\bf t}$, $G$, $K_s(f')$, $S_{\rm sh}(f')$, $\mu(f,f')$, and $\mathcal{A}(f,f')$ defined in the previous subsection.

\begin{proposition}\label{free-v-Pro}
Let $(\rho^-,{\bf u}^-,p^-)$ be from Problem \ref{S-div-Prob}.
For notational simplicity, let us set 
{\small \begin{equation}\label{sigma-def-1}
\sigma(\rho^-,{\bf u}^-,p^-,v_{\rm ex}):=\|(\rho^-,{\bf u}^-,p^-)-(\rho_0^-,\nabla\varphi_0^-,p_0^-)\|_{2,\alpha,\Om}+\|v_{\rm ex}-v_c\|^{(-\alpha,\partial\Gamma_{\rm ex})}_{1,\alpha,\Gamma_{\rm ex}}.
\end{equation}}
\begin{itemize}
\item[(a)] For any $\alpha\in(\frac{1}{2},1)$, there exists a small constant $\sigma_2^{\ast}>0$ depending only on the data so that if
\begin{equation*}
\sigma(\rho^-,{\bf u}^-,p^-,v_{\rm ex})\le \sigma_2^{\ast},
\end{equation*}
then there exists an axisymmetric solution $(f,S,\Lambda,\varphi,{\bf W})$ of the free boundary problem \eqref{free-v} that satisfies the estimate 
\begin{equation}\label{free-v-est}
\begin{split}
&\|f-r_{\rm sh}\|_{2,\alpha,(0,\phi_0)}^{(-1-\alpha,\{\phi=\phi_0\})}
+\|(S,\frac{\Lambda}{r\sin\phi}{\bf e}_{\theta})-(S_0^+,{\bf 0})\|_{1,\alpha,\Om_f^+}^{(-\alpha,\Gamma_{\rm w})}\\
&+\|\varphi-\varphi_0^+\|_{2,\alpha,\Om_f^+}^{(-1-\alpha,\Gamma_{\rm w})}+\|{\bf W}\|_{2,\alpha,\Om_f^+}^{(-1-\alpha,\Gamma_{\rm w})}\le C\sigma(\rho^-,{\bf u}^-,p^-,v_{\rm ex})
\end{split}
\end{equation}
for a constant $C>0$ depending only on the data.
Moreover, ${\bf W}$ has of the form ${\bf W}=\psi{\bf e}_{\theta}$ for an axisymmetric function $\psi$.
\item[(b)] 
There exists a small constant $\tilde{\sigma}_2^{\ast}\in(0,\sigma_2^{\ast}]$ so that if 
$$\sigma(\rho^-,{\bf u}^-,p^-,v_{\rm ex})\le\tilde{\sigma}_2^{\ast},$$ then 
the solution of \eqref{free-v} that satisfies the estimate \eqref{free-v-est} is unique.
\end{itemize}
\end{proposition}
An overview of the proof of Proposition \ref{free-v-Pro} is provided in  the next subsection, and the proof is completed in \S\ref{Section-Pro}. 

For notational simplicity, we define neighborhoods of $(\rho_0^-,\varphi_0^-,p_0^-,v_c)$ and $r_{\rm sh}$ with radius $R>0$ as follows: 
{\small \begin{equation}
\begin{split}
&\mathcal{B}_R(\rho_0^-,\varphi_0^-,p_0^-):=\left\{\begin{split}
										&(\rho^-,\varphi^-,p^-)\in C^{2,\alpha}(\overline{\Om})\times C^{3,\alpha}(\overline{\Om})\times C^{2,\alpha}(\overline{\Om}):\\
										&\rho^-,\varphi^-,p^-\mbox{ are axisymmetric},\\
										&\|\rho^--\rho_0^-\|_{2,\alpha,\Om}+\|\varphi^--\varphi_0^-\|_{3,\alpha,\Om}+\|p^--p_0^-\|_{2,\alpha,\Om}\le R\\
										\end{split}\right\},\\
&\mathcal{B}_R(v_c):=\left\{v_{\rm ex}\in C^{1,\alpha}_{(-\alpha,\partial\Gamma_{\rm ex})}({\Gamma_{\rm ex}}): 
		\begin{split}
		&v_{\rm ex}\mbox{ is axisymmetric},\\
		&\|v_{\rm ex}-v_c\|_{1,\alpha,\Gamma_{\rm ex}}^{(-\alpha,\partial\Gamma_{\rm ex})}\le R\end{split}\right\},\\
&\mathcal{B}_R(\rho_0^-,\varphi_0^-,p_0^-,v_c):=\mathcal{B}_R(\rho_0^-,\varphi_0^-,p_0^-)\times\mathcal{B}_R(v_c),\\
&\mathcal{B}_R(r_{\rm sh}):=\left\{f\in C_{(-1-\alpha,\{\phi=\phi_0\})}^{2,\alpha}((0,\phi_0)):\,\|f-r_{\rm sh}\|_{2,\alpha,(0,\phi_0)}^{(-1-\alpha,\{\phi=\phi_0\})}\le R\right\}.
\end{split}
\end{equation}}
Note that, for sufficiently small $R>0$, $\mathcal{B}_R(\rho_0^-,\varphi_0^-,p_0^-,v_c)$ and $\mathcal{B}_R(r_{\rm sh})$ are sets of small perturbations of $(\rho_0^-,\varphi_0^-,p_0^-,v_c)$ and $r_{\rm sh}$, respectively.

For positive constants $\sigma$ and $C$, 
define a mapping $\mathcal{P}:\mathcal{B}_{\sigma}(\rho_0^-,\varphi_0^-,p_0^-,v_c)\to\mathcal{B}_{C\sigma}(p_0^+)$ by 
$$\mathcal{P}:(\rho^-,\varphi^-,p^-,v_{\rm ex})\mapsto \left.S\varrho^{\gamma}\left(S,\nabla\varphi+\nabla\times{\bf W}+\left(\frac{\Lambda}{r\sin\phi}\right){\bf e}_{\theta}\right)\right\vert_{\Gamma_{\rm ex}},$$
 where $(f,S,\Lambda,\varphi,{\bf W})$ is an axisymmetric solution to the free boundary problem \eqref{free-v} associated with $(\rho^-,\varphi^-,p^-,v_{\rm ex})$.
Once Proposition \ref{free-v-Pro} is proved, there exist positive constants $\sigma\in(0,\tilde{\sigma}_{2}^{\ast}]$ and $C$ so that the mapping $\mathcal{P}$ is well-defined.
In \S\ref{S-exi-Thm}, we prove that $\mathcal{P}$ is locally invertible in a weak sense near the background solution.
Then, for a given $p_{\rm ex}$, there exists  $v_{\rm ex}$ such that $\mathcal{P}(v_{\rm ex})=p_{\rm ex}$ and 
\begin{equation}\label{v-vc}
\|v_{\rm ex}-v_c\|_{1,\alpha,\Gamma_{\rm ex}}^{(-\alpha,\partial\Gamma_{\rm ex})}\le C\sigma(\rho^-,{\bf u}^-,p^-,p_{\rm ex}).
\end{equation}
Since  the solution $(f,S,\Lambda,\varphi,{\bf W})$ of the free boundary problem \eqref{free-v} satisfies \eqref{E-re}, \eqref{bd-W}, and \eqref{boundary-HD}, 
it also satisfies \eqref{E-re} with \eqref{bd-W}-\eqref{bd-EX} and \eqref{boundary-HD}.
Moreover, by \eqref{free-v-est} and \eqref{v-vc}, the solution $(f,S,\Lambda,\varphi,{\bf W})$  satisfies the estimate \eqref{Thm-HD-est}.

Finally, we prove the uniqueness of $v_{\rm ex}$ satisfying $\mathcal{P}(v_{\rm ex})=p_{\rm ex}$ by a contraction argument. 
Then, the uniqueness of $v_{\rm ex}$ implies that the solution $(f,S,\Lambda,\varphi,{\bf W})$ of the free boundary problem \eqref{free-v} associated with $v_{\rm ex}$ is the desired unique solution. 
The details are given in \S\ref{S-uni-Thm}.


\subsection{Outline of the proof of Proposition \ref{free-v-Pro}}\label{Ss-pro}
For notational simplicity, let us set $\sigma_v$ as
\begin{equation}
\sigma_v:=\sigma(\rho^-,{\bf u}^-,p^-,v_{\rm ex})
\end{equation}
for $\sigma(\rho^-,{\bf u}^-,p^-,v_{\rm ex})$ given in \eqref{sigma-def-1}.

{\bf Step 1.} Fix $\alpha\in(\frac{1}{2},1)$ and set $\Om^+:=\Om\cap\{r>r_{\rm sh}-\frac{1}{2}\}$. 
For a positive constant $\delta_1$  to be determined later, define an iteration set $\mathcal{I}(\delta_1)$  by 
\begin{equation}
\mathcal{I}(\delta_1):=\left\{{\bf W}=\psi{\bf e}_{\theta}\in C^{2,\alpha}_{(-1-\alpha,\Gamma_{\rm w})}({\Om^+};\mathbb{R}^3)\;\middle\vert\;\begin{split}
							&\psi \mbox{ is axisymmetric},\\
							&\psi\equiv0\mbox{ on }\{\phi=0,\phi_0\},\\
							&\|{\bf W}\|_{2,\alpha,\Om^+}^{(-1-\alpha,\Gamma_{\rm w})}\le \delta_1\sigma_v
							\end{split}\right\}.
\end{equation}
For a fixed ${\bf W}_{\ast}=\psi_{\ast}{\bf e}_{\theta}\in\mathcal{I}(\delta_1)$, let us set 
${\bf t}_{\ast}:=\nabla\times{\bf W}_{\ast}+\left(\frac{\Lambda}{r\sin\phi}\right){\bf e}_{\theta}$
for notational simplicity.
We solve the following free boundary problem for $(f,\varphi,S,\Lambda)$:
\begin{equation}\label{W-free}
\begin{split}
&\left\{\begin{split}
	&\mbox{div}\left({\bf A}(\nabla\varphi,{\bf t}_{\ast})\right)=0\\
	&{\bf A}(\nabla\varphi,{\bf t}_{\ast})\cdot\nabla S=0\\
	&{\bf A}(\nabla\varphi,{\bf t}_{\ast})\cdot\nabla \Lambda=0	
	\end{split}\right.\quad\mbox{in}\quad\Om_f^+,\\
&\left\{\begin{split}
	&S=S_{\rm sh}(f'),\quad \Lambda=\Lambda^-\quad\mbox{on}\quad\mathfrak{S}_f,\\
		&\varphi=\varphi^-,\quad-\nabla\varphi\cdot{\bf n}_f=-\frac{K_s(f')}{{\bf u}^-\cdot{\bf n}_f}+{\bf t}_{\ast}\cdot{\bf n}_f\quad\mbox{on}\quad\mathfrak{S}_f,\\
	&\nabla\varphi\cdot{\bf e}_{\phi}=0\quad\mbox{on}\quad\Gamma_{\rm w}\cap\partial\Om_f^+,\\
	&\left(B_0-\frac{\lvert\nabla\varphi+{\bf t}_{\ast}\rvert^2}{2}\right)^{\frac{1}{\gamma-1}}\nabla\varphi\cdot{\bf e}_r=v_{\rm ex}\quad\mbox{on}\quad\Gamma_{\rm ex}.
	\end{split}\right.
\end{split}
\end{equation}
\begin{lemma}\label{Pro1}
For any $\alpha\in(\frac{1}{2},1)$, there exists a small constant $\sigma_3>0$ depending only on the data and $\delta_1$ 
so that if
\begin{equation*}
\sigma_v\le \sigma_3,
\end{equation*}
then there exists a unique axisymmetric solution $(f,S,\Lambda,\varphi)$ of the free boundary problem \eqref{W-free} that satisfies the estimate 
\begin{equation}\label{Pro-est}
\begin{split}
&\|f-r_{\rm sh}\|_{2,\alpha,(0,\phi_0)}^{(-1-\alpha,\{\phi=\phi_0\})}+\|(S,\frac{\Lambda}{r\sin\phi}{\bf e}_{\theta})-(S_0^+,{\bf 0})\|_{1,\alpha,\Om_f^+}^{(-\alpha,\Gamma_{\rm w})}\\
&+\|\varphi-\varphi_0^+\|_{2,\alpha,\Om_f^+}^{(-1-\alpha,\Gamma_{\rm w})}\le C\left(1+\delta_1\right)\sigma_v
\end{split}
\end{equation}
for a constant $C>0$ depending only on the data. 
Furthermore, $S$ satisfies 
\begin{equation}\label{Pro1-S-est}
\|\partial_\phi S\|_{0,\alpha,\Om_f^+}^{(1-\alpha,\Gamma_{\rm w})}\le C\left(\sigma_v+\delta_1^2\sigma_v^2\right)
\end{equation}
for a constant $C>0$ depending only on the data.
\end{lemma}
The proof of Lemma \ref{Pro1} is given in the next section.



{\bf Step 2.} Let $(f,S_{\sharp},\Lambda_{\sharp},\varphi_{\sharp})$ be the solution obtained in  Lemma \ref{Pro1}, and let us  set 
\begin{equation*}
{\bf t}_{\sharp}:=\nabla\times{\bf W}_{\ast}+\left(\frac{\Lambda_{\sharp}}{r\sin\phi}\right){\bf e}_{\theta}.
\end{equation*}
We solve the following linear boundary value problem for ${\bf W}$:
\begin{equation}\label{lin-W}
\left\{\begin{split}
&-\Delta{\bf W}=G_{\sharp}{\bf e}_{\theta}\quad\mbox{in}\quad\Om_f^+,\\
	&-\nabla{\bf W}\cdot{\bf n}_{f}+\mu(f,f'){\bf W}=\mathcal{A}(f,f'){\bf e}_{\theta}\quad\mbox{on}\quad\mathfrak{S}_f,\\
	&{\bf W}={\bf 0}\quad\mbox{on}\quad\partial\Om_f^+\backslash\mathfrak{S}_f
\end{split}\right.
\end{equation}
with
\begin{equation}\label{def-G-sharp}
G_{\sharp}:=G\left(r,\phi,S_{\sharp},\Lambda_{\sharp},\partial_{\phi}S_{\sharp},\partial_{\phi}\Lambda_{\sharp},\nabla\varphi_{\sharp}+{\bf t}_{\sharp}\right)
\end{equation}
for $G$, $\mu$, and $\mathcal{A}$ given in \eqref{def-H-G} and \eqref{def-mu}.
\begin{lemma}\label{Pro2}
For any $\alpha\in(\frac{1}{2},1)$, there exists a small constant $\sigma_4>0$ depending only on the data and $\delta_1$ 
so that if
\begin{equation*}
\sigma_v
\le \sigma_4,
\end{equation*}
then there exists a unique solution ${\bf W}$ of the linear boundary value problem \eqref{lin-W} that satisfies the estimate 
\begin{equation}\label{W-est}
\begin{split}
\|{\bf W}\|_{2,\alpha,\Om_f^+}^{(-1-\alpha,\Gamma_{\rm w})}\le&\, C\left(\|G_{\sharp}{\bf e}_{\theta}\|_{0,\alpha,\Om_f^+}^{(1-\alpha,\Gamma_{\rm w})}+\|\mathcal{A}(f,f'){\bf e}_{\theta}\|_{1,\alpha,\mathfrak{S}_f}^{(-\alpha,\partial\mathfrak{S}_f)}\right)\\
\le&\, C\left(\sigma_v+\delta_1^2\sigma_v^2\right)
\end{split}
\end{equation}
for a constant $C>0$ depending only on the data. 
Furthermore, ${\bf W}$ has of the form 
\begin{equation*}
{\bf W}=\psi(x,r){\bf e}_{\theta}
\end{equation*}
for an axisymmetric function $\psi$ satisfying 
\begin{equation}\label{psi-eqn}
\left\{\begin{split}
&-\left(\Delta_{\bf x}-\frac{1}{r^2\sin^2\phi}\right)\psi=G(r,\phi,S_{\sharp},\Lambda_{\sharp},\partial_{\phi}S_{\sharp},\partial_{\phi}\Lambda_{\sharp},\nabla\varphi_{\sharp}+{\bf t}_{\sharp})\,\,\mbox{in }\Om_f^+,\\
&-\nabla \psi\cdot{\bf n}_f+\mu(f,f')\psi=\mathcal{A}(f,f')\,\,\mbox{on }\mathfrak{S}_f,\\
&\psi=0\,\,\mbox{on }\partial\Om_f^+\backslash\mathfrak{S}_f.
\end{split}\right.
\end{equation}
\end{lemma}
The proof of Lemma \ref{Pro2} can be obtained by  slightly adjusting \cite[Proof of Theorem 3.1]{park2020transonic}. So we skip it.
In order to set an iteration map for ${\bf W}$, we need to extend the solution ${\bf W}$ into $\left[C^{2,\alpha}_{(-1-\alpha,\Gamma_{\rm w})}({\Om^+})\right]^3$. 
For that purpose, we first define a transformation and a corresponding Jacobian matrix as follows:
\begin{definition}
For $f,g\in\mathfrak{B}_{\frac{1}{4}}(r_{\rm sh})$, define a transformation $\mathfrak{T}_{f,g}:\Om_f^+\to\Om_g^+$ by 
\begin{equation}\label{no-trans}
\begin{split}
\mathfrak{T}_{f,g}:
&\,{\bf x}\mapsto\left(\frac{r_{\rm ex}-g}{r_{\rm ex}-f}(\lvert{\bf x}\rvert-r_{\rm ex})+r_{\rm ex}\right)\frac{\bf x}{\lvert{\bf x}\rvert},\\
&(r,\phi,\theta)\mapsto\left(\frac{r_{\rm ex}-g(\phi)}{r_{\rm ex}-f(\phi)}(r-r_{\rm ex})+r_{\rm ex},\phi,\theta\right).
\end{split}
\end{equation}
Since $\lvert r_{\rm ex}-f\rvert>0$ and $\lvert r_{\rm ex}-g\rvert>0$, $\mathfrak{T}_{f,g}$ is invertible and $\mathfrak{T}_{f,g}^{-1}=\mathfrak{T}_{g,f}$.
For a transformation $\mathfrak{T}_{f,g}:{\bf x}\mapsto{\bf y}$ defined by \eqref{no-trans}, define a Jacobian matrix $\mathbf{J}_{f,g}$ by
\begin{equation}\label{def-jaco}
{\bf J}_{f,g}:=[\partial_{x_i}y_j]_{i,j=1}^3.
\end{equation}
\end{definition}

Define a function ${\bf W}_{\rm e}$ by 
\begin{equation}\label{def-We}
{\bf W}_{\rm e}({\bf y}):=\left\{\begin{split}
	&{\bf W}\circ\mathfrak{T}_{r_{\rm sh},f}({\bf y})\mbox{ for }{\bf y}\in\overline{\mathfrak{T}_{f,r_{\rm sh}}(\Om^+)}\cap\{\tilde{r}>r_{\rm sh}\},\\
	&\sum_{i=1}^3c_i\left({\bf W}\circ\mathfrak{T}_{r_{\rm sh},f}\right)\left(-\frac{\tilde{r}}{i},\phi,\theta\right)\mbox{ for }{\bf y}\in\overline{\mathfrak{T}_{f,r_{\rm sh}}(\Om^+)}\backslash\{\tilde{r}>r_{\rm sh}\},
	\end{split}
\right.
\end{equation}
where $(\tilde{r},\phi,\theta)$ denote the spherical coordinates of ${\bf y}=(y_1,y_2,y_3)\in\overline{\mathfrak{T}_{f,r_{\rm sh}}(\Om^+)}$, that is,
$(y_1,y_2,y_3)=(\tilde{r}\sin\phi\cos\theta,\tilde{r}\sin\phi\sin\theta,\tilde{r}\cos\theta).$
In \eqref{def-We}, $(c_1,c_2,c_3)$ is the solution of the system 
$\sum_{i=1}^3 c_i\left(-\frac{1}{i}\right)^m=1,$ $m=0,1,2.$
With such a  function ${\bf W}_e$, 
we define a function ${\bf W}_{\rm ext}:\Om^+\to\mathbb{R}^3$ by 
\begin{equation}\label{W-ext}
{\bf W}_{\rm ext}({\bf x}):={\bf W}_{\rm e}\circ\mathfrak{T}_{f,r_{\rm sh}}({\bf x}).
\end{equation}
Obviously, the function ${\bf W}_{\rm ext}$  satisfies the estimate 
\begin{equation*}\label{Wext-est}
\|{\bf W}_{\rm ext}\|_{2,\alpha,\Om^+}^{(-1-\alpha,\Gamma_{\rm w})}\le C_{\flat}\left(\sigma_v+\delta_1^2\sigma_v^2\right)
\end{equation*}
for a constant $C_{\flat}>0$ depending only on the data. 

 Now we define an iteration mapping $\mathcal{J}_v:\mathcal{I}(\delta_1)\to \left[C^{2,\alpha}_{(-1-\alpha,\Gamma_{\rm w})}({\Om^+})\right]^3$ by 
 \begin{equation*}
 \mathcal{J}_v:{\bf W}_{\ast}\mapsto {\bf W}_{\rm ext}.
 \end{equation*}
Then we choose $\delta_1$ and $\sigma_2^{\ast}$ satisfying
 \begin{equation}\label{del1-ch}
 \delta_1=2C_{\flat}\quad\mbox{and}\quad\sigma_2^{\ast}\le\min\left\{\sigma_3,\sigma_4,\frac{1}{2C_{\flat}\delta_1},\frac{1}{4C(1+\delta_1)}\right\}
 \end{equation}
 for $C$ in \eqref{Pro-est}
  so that the mapping $\mathcal{J}_v$ maps $\mathcal{I}(\delta_1)$ into itself. 
The iteration set $\mathcal{I}(\delta_1)$ is a convex and compact subset of $\left[C^{2,\alpha/2}_{(-1-\alpha/2,\Gamma_{\rm w})}({\Om^+})\right]^3$, and  $\mathcal{J}_v$ is continuous in $\left[C^{2,\alpha/2}_{(-1-\alpha/2,\Gamma_{\rm w})}({\Om^+})\right]^3$.
Hence, by the Schauder fixed point theorem, $\mathcal{J}_v$ has a fixed point in $\mathcal{I}(\delta_1)$, say ${\bf W}_{\natural}$.

According to Lemma \ref{Pro1}, the free boundary value problem \eqref{W-free} associated with ${\bf W}_{\natural}$ has a unique solution  $(f_{\natural},S_{\natural},\Lambda_{\natural},\varphi_{\natural})$ satisfying \eqref{Pro-est}-\eqref{Pro1-S-est}.
Then, obviously, $(f_{\natural},S_{\natural},\Lambda_{\natural},\varphi_{\natural},{\bf W}_{\natural})$ becomes a solution to the  free boundary problem \eqref{free-v} that satisfies the estimate \eqref{free-v-est}.
Finally, we prove the uniqueness of a solution by using a contraction argument  in \S\ref{Section-Pro}.


\section{Proof of Lemma \ref{Pro1}}\label{Sec-Pro}
\subsection{Proof of Lemma \ref{Pro1}}
Fix ${\bf W}_{\ast}\in\mathcal{I}(\delta_1)$ and set 
${\bf t}_{\ast}:=\nabla\times{\bf W}_{\ast}+\left(\frac{\Lambda}{r\sin\phi}\right){\bf e}_{\theta}.$ 
In this subsection, we solve the following free boundary problem for $(f,\varphi,S,\Lambda)$:
\begingroup
\begin{align*}\allowdisplaybreaks
\begin{aligned}
&\left\{\begin{aligned}
	&\mbox{div}\left({\bf A}(\nabla\varphi,{\bf t}_{\ast})\right)=0\quad\mbox{in}\quad\Om_f^+,\\
	&\varphi=\varphi^-,\quad-\nabla\varphi\cdot{\bf n}_f=-\frac{K_s(f')}{{\bf u}^-\cdot{\bf n}_f}+{\bf t}_{\ast}\cdot{\bf n}_f\quad\mbox{on}\quad\mathfrak{S}_f,\\
	&\nabla\varphi\cdot{\bf e}_{\phi}=0\quad\mbox{on}\quad\Gamma_{\rm w}\cap\partial\Om_f^+,\\
	&\left(B_0-\frac{\lvert\nabla\varphi+{\bf t}_{\ast}\rvert^2}{2}\right)^{\frac{1}{\gamma-1}}\nabla\varphi\cdot{\bf e}_r=v_{\rm ex}\quad\mbox{on}\quad\Gamma_{\rm ex},
	\end{aligned}\right.\\
&\left\{\begin{aligned}
	&\begin{aligned}
	&{\bf A}(\nabla\varphi,{\bf t}_{\ast})\cdot\nabla S=0,\quad
	{\bf A}(\nabla\varphi,{\bf t}_{\ast})\cdot\nabla \Lambda=0
	\end{aligned}\quad\mbox{in}\quad\Om_f^+,\\
	&S=S_{\rm sh}(f'),\quad \Lambda=\Lambda^-\quad\mbox{on}\quad\mathfrak{S}_f
	\end{aligned}\right.
\end{aligned}
\end{align*}
\endgroup
for ${\bf A}$, $K_s(f')$, $v_{\rm ex}$, and $S_{\rm sh}(f')$ given in \eqref{def-H-G}, \eqref{def-KS}, \eqref{v-c-def}, and \eqref{def-A}, respectively.

For a positive constant $\delta_2$ to be determined later, define an iteration set $\mathcal{I}(\delta_2)$  by
\begin{equation}\label{ite-set-2}
\mathcal{I}(\delta_2):=\left\{\Lambda\in C^{1,\alpha}_{(-\alpha,\Gamma_{\rm w})}({\Om^+})\,\middle\vert\,
						\begin{split}
						&\Lambda\mbox{ is axisymmetric, }\\
						&\Lambda\equiv0\mbox{ on }\Om^+\cap\{\phi=0\},\\
						&\left\|\frac{\Lambda}{\sin\phi}\right\|_{1,\alpha,\mathcal{R}^+}^{(-\alpha,\{\phi=\phi_0\})}
						\le \delta_2\sigma_v
						\end{split}\right\},
\end{equation}
where $\mathcal{R}^+$ is a two-dimensional rectangular domain defined by $$\mathcal{R}^+:=\left\{({r},\phi)\in\mathbb{R}^2:\, r_{\rm sh}-\frac{1}{2}<r<r_{\rm ex},0<\phi<\phi_0\right\}.$$
\begin{remark}
If $\Lambda$ satisfies the estimate in \eqref{ite-set-2}, then one can directly check that $\Lambda$ satisfies
\begin{equation*}
\left\|\left(\frac{\Lambda}{r\sin\phi}\right){\bf e}_{\theta}\right\|_{1,\alpha,\Om^+}^{(-\alpha,\Gamma_{\rm w})}\le C\delta_2\sigma_v\mbox{ and }
\left\|\left(\frac{\Lambda\partial_{\phi}\Lambda}{\sin^2\phi}\right){\bf e}_{\theta}\right\|_{0,\alpha,\Om^+}^{(1-\alpha,\Gamma_{\rm w})}\le C(\delta_2\sigma_v)^2.
\end{equation*}
\end{remark}
For a fixed $\Lambda_{\ast}\in\mathcal{I}(\delta_2)$, let us set  
${\bf t}_{\star}:=\nabla\times{\bf W}_{\ast}+\left(\frac{\Lambda_{\ast}}{r\sin\phi}\right){\bf e}_{\theta}.$ 
We first solve the following free boundary problem for $(f,\varphi)$:
\begin{equation}\label{free-var}
\left\{\begin{split}
	&\mbox{div}\left({\bf A}(\nabla\varphi,{\bf t}_{\star})\right)=0\quad\mbox{in}\quad\Om_f^+,\\
	&\varphi=\varphi^-,\quad-\nabla\varphi\cdot{\bf n}_f=-\frac{K_s(f')}{{\bf u}^-\cdot{\bf n}_f}+{\bf t}_{\star}\cdot{\bf n}_f\quad\mbox{on}\quad\mathfrak{S}_f,\\
	&\nabla\varphi\cdot{\bf e}_{\phi}=0\quad\mbox{on}\quad\Gamma_{\rm w}\cap\partial\Om_f^+,\\
	&\left(B_0-\frac{\lvert \nabla\varphi+{\bf t}_{\star}\rvert^2}{2}\right)^{\frac{1}{\gamma-1}}\nabla\varphi\cdot{\bf e}_r=v_{\rm ex}\quad\mbox{on}\quad\Gamma_{\rm ex}.
	\end{split}\right.
\end{equation}
\begin{lemma}\label{Lem-f-var}
For any $\alpha\in(\frac{1}{2},1)$, there exists a small positive constant $\sigma_5$ depending only on the data and $(\delta_1,\delta_2)$ so that if 
$$\sigma_v\le\sigma_5,$$
then the free boundary problem \eqref{free-var} has a unique axisymmetric solution $(f,\varphi)$ that satisfies 
\begin{equation}\label{Lem-f-est}
\|f-r_{\rm sh}\|_{2,\alpha,(0,\phi_0)}^{(-1-\alpha,\{\phi=\phi_0\})}+\|\varphi-\varphi_0^+\|_{2,\alpha,\Om_f^+}^{(-1-\alpha,\Gamma_{\rm w})}
\le C\left(1+\delta_1+\delta_2\right)\sigma_v
\end{equation}
for a constant $C>0$ depending only on the data.
\end{lemma}
The proof of this lemma is provided in the next subsection. 
%
For  $(f,\varphi)$ obtained in Lemma \ref{Lem-f-var}, consider the following initial value problem for $(\Lambda,S)$:
\begin{equation}\label{ini-S}
\left\{\begin{split}
	&\begin{split}
	&{\bf A}(\nabla\varphi,{\bf t}_{\star})\cdot\nabla \Lambda=0,\quad
	{\bf A}(\nabla\varphi,{\bf t}_{\star})\cdot\nabla S=0\\
	\end{split}\quad\mbox{in}\quad\Om_f^+,\\
	&\Lambda=\Lambda^-,\quad S=S_{\rm sh}(f')\quad\mbox{on}\quad\mathfrak{S}_f.
\end{split}\right.
\end{equation}
\begin{lemma}\label{Lem-t}
For any $\alpha\in(\frac{1}{2},1)$,
there exists a small positive constant $\sigma_6\in(0,\sigma_5]$ depending only on the data and $(\delta_1,\delta_2)$  so that if 
$$\sigma_v\le\sigma_6,$$
then the initial value problem \eqref{ini-S} has a unique axisymmetric solution $(\Lambda,S)$ that satisfies 
\begin{equation}\label{R-Lambda}
\left\|\frac{\Lambda}{\sin\phi}\right\|_{1,\alpha,\mathcal{R}_f^+}^{(-\alpha,\{\phi=\phi_0\})}\le C\sigma_v\quad\mbox{and}\quad
\|S-S_0^+\|_{1,\alpha,\Om_f^+}^{(-\alpha,\Gamma_{\rm w})}\le C\left(1+\delta_1+\delta_2\right)\sigma_v
\end{equation}
for a constant $C>0$ depending only on the data. In \eqref{R-Lambda}, $\mathcal{R}_f^+$ is a two-dimensional space defined by 
$\mathcal{R}_f^+:=\{(r,\phi)\in\mathbb{R}^2:\, f(\phi)<r<r_{\rm ex}, \phi\in(0,\phi_0)\}$.
Furthermore, $S$ satisfies 
\begin{equation*}
\|\partial_\phi S\|_{0,\alpha,\Om_f^+}^{(1-\alpha,\Gamma_{\rm w})}\le C\left(1+\delta_1^2\sigma_v+\delta_2^2\sigma_v\right)\sigma_v
\end{equation*}
for a constant $C>0$ depending only on the data.
\end{lemma}
One can prove Lemma \ref{Lem-t} by slightly adjusting \cite[Proof of Lemma 4.3]{park2020transonic}.  So we skip it. The key point of the proof is that the entropy $S$ and the angular momentum density $\Lambda$ are conserved along each stream line.

For the solution $\Lambda$ of \eqref{ini-S}, consider an extension $\Lambda_{\rm ext}$ defined in the same way as the definition \eqref{W-ext} of ${\bf W}_{\rm ext}$.
Then, $\Lambda_{\rm ext}$ satisfies 
\begin{equation*}
\begin{split}
\left\|\frac{\Lambda_{\rm ext}}{\sin\phi}\right\|_{1,\alpha,\mathcal{R}^+}^{(-\alpha,\{\phi=\phi_0\})}\le C_{\star}\sigma_v
\end{split}
\end{equation*}
for a constant $C_{\star}>0$ depending only on the data. 
We choose $\delta_2$ as 
\begin{equation*}
\delta_2:=C_{\star}
\end{equation*}
so that an iteration mapping $\mathcal{J}_t:\mathcal{I}(\delta_2)\to C^{1,\alpha}_{(-\alpha,\Gamma_{\rm w})}({\Om^+})$ defined by 
\begin{equation*}
\mathcal{J}_t:\Lambda_{\ast}\mapsto\Lambda_{\rm ext}
\end{equation*}
 maps $\mathcal{I}(\delta_2)$ into itself.
Since the iteration set $\mathcal{I}(\delta_2)$ is a convex and compact subset of $C^{1,\alpha/2}_{(-\alpha/2,\Gamma_{\rm w})}({\Om^+})$ and 
 $\mathcal{J}_t$ is continuous in $C^{1,\alpha/2}_{(-\alpha/2,\Gamma_{\rm w})}({\Om^+})$, 
 the Schauder fixed point theorem implies the mapping $\mathcal{J}_t$ has a fixed point $\Lambda_{\natural}\in\mathcal{I}(\delta_2)$. 
According to Lemma \ref{Lem-f-var}, there exists a unique solution $(f_{\natural},\varphi_{\natural})$ to the free boundary problem \eqref{free-var} associated with $\Lambda_{\ast}=\Lambda_{\natural}$.
And, according to Lemma \ref{Lem-t}, there exists a unique solution $S_{\natural}$ to the initial value problem \eqref{ini-S} associated with $(f,\varphi,\Lambda_{\ast})=(f_{\natural},\varphi_{\natural},\Lambda_{\natural})$.
Then, $(f_{\natural},S_{\natural}, \Lambda_{\natural}\vert_{\Om_{f_{\natural}}^+}, \varphi_{\natural})$ solves the free boundary problem \eqref{W-free} and satisfies the estimate \eqref{Pro-est}-\eqref{Pro1-S-est}.

To complete the proof, it remains to prove the uniqueness of a solution. 
The following lemma can be verified by the standard contraction argument.
\begin{lemma}\label{Uni-ss}
Let $\mathcal{U}^{(k)}:=(f^{(k)}, S^{(k)},\Lambda^{(k)}, \varphi^{(k)})$ ($k=1,2)$ be two axisymmetric solutions to the free boundary problem \eqref{W-free}, and suppose that each solution satisfies the estimate \eqref{Pro-est}-\eqref{Pro1-S-est}.
There exists a small positive constant $\sigma_7\in(0,\sigma_6]$ depending only on the data and $\delta_1$  so that if 
$$\sigma_v\le\sigma_7,$$
then  $\mathcal{U}^{(1)}=\mathcal{U}^{(2)}$.
\end{lemma}
%
Finally, we choose $\sigma_3$ as $\sigma_3=\sigma_7$ to complete the proof of Lemma \ref{Pro1}.\qed

\subsection{Proof of Lemma \ref{Lem-f-var}: Free boundary problem for $(f,\varphi)$}\label{S=lem4.2}
We solve  the following free boundary problem for $(f,\varphi)$:
\begin{eqnarray}
	&\label{re-Lem-41}&\mbox{div}\left({\bf A}(\nabla\varphi,{\bf t}_{\star})\right)=0\quad\mbox{in}\quad\Om_f^+,\\
	&\label{chi-bd-s}&\varphi=\varphi^-,\quad-\nabla\varphi\cdot{\bf n}_f=-\frac{K_s(f')}{{\bf u}^-\cdot{\bf n}_f}+{\bf t}_{\star}\cdot{\bf n}_f\quad\mbox{on}\quad\mathfrak{S}_f,\\
	&\label{chi-bd-w}&\nabla\varphi\cdot{\bf e}_{\phi}=0\quad\mbox{on}\quad\Gamma_{\rm w}\cap\partial\Om_f^+,\\
	&\label{chi-bd-ex}&\left(B_0-\frac{\lvert\nabla\varphi+{\bf t}_{\star}\rvert^2}{2}\right)^{\frac{1}{\gamma-1}}\nabla\varphi\cdot{\bf e}_r=v_{\rm ex}\quad\mbox{on}\quad\Gamma_{\rm ex}
	\end{eqnarray}
for ${\bf A}$, $K_s(f')$, and $v_{\rm ex}$ defined by \eqref{def-H-G}, \eqref{def-KS}, and \eqref{v-c-def}, respectively.

{\bf Step 1.} (Linearization) 
Set $\chi:=\varphi-\varphi_0^+$.
For  ${\bf s}, {\bf t}\in\mathbb{R}^3$, let us set 
\begin{equation}\label{def-mathfrak-a}
\mathfrak{a}({\bf s},{\bf t}):=\left(B_0-\frac{\lvert{\bf s}+{\bf t}\rvert^2}{2}\right)^{\frac{1}{\gamma-1}}\quad\mbox{and}\quad \tilde{\bf A}({\bf s},{\bf t}):=\mathfrak{a}({\bf s},{\bf t}){\bf s}.
\end{equation}
Then the equation $\mbox{div}\left({\bf A}(\nabla\varphi,{\bf t})\right)=0$ can be rewritten as 
\begin{equation}\label{chi-aij-eq}
\begin{split}
\sum_{i,j=1}^3\partial_j(a_{ij}({\bf 0})\partial_i\chi)
=&\,\mbox{div}{\bf F}(\nabla\chi,{\bf t}),
\end{split}
\end{equation}
where $a_{ij}$ and ${\bf F}=(F_1,F_2,F_3)$ are defined by 
\begin{equation}\label{Def-F}
\begin{split}
a_{ij}({\bf q}):=&\, \int_0^1\partial_{s_i}A_j(V_0+\eta{\bf q})d\eta\quad\mbox{for }A_j:=\tilde{\bf A}\cdot{\bf e}_j,\,V_0:=(\nabla\varphi_0^+,{\bf 0}),\\
F_i({\bf q}):=&\,-\mathfrak{a}(V_0+{\bf q})t_i-\int_0^1D_{\bf t}A_i(V_0+\eta{\bf q})d\eta\cdot {\bf t}\\
&-\int_0^1 D_{\bf s}A_i(V_0+\eta{\bf q})-D_{\bf s}A_i(V_0)d\eta\cdot{\bf s}
\end{split}
\end{equation}
with ${\bf q}=({\bf s},{\bf t})\in\mathbb{R}^3\times\mathbb{R}^3$, ${\bf s}=(s_1,s_2,s_3)$, ${\bf t}=(t_1,t_2,t_3)$.

Now, we derive the boundary conditions for $\chi$ from  \eqref{chi-bd-s}-\eqref{chi-bd-ex}.
The first boundary condition in \eqref{chi-bd-s} can be rewritten as
	\begin{equation}\label{c1}
	\chi+\varphi_0^+-\varphi^-=0\quad\mbox{on }\mathfrak{S}_f.
	\end{equation}
	A direct computation yields 
	\begin{equation}\label{shock-bd-h1}
	(a_{ij}({\bf 0})\partial_i\chi)\cdot({-\bf n}_f)=h_1(\nabla\chi,f')\mbox{ on }\mathfrak{S}_f
	\end{equation}
	for $h_1$ defined by 
	\begin{equation}\label{def-h1}
\begin{split}
	h_1(\nabla\chi,f')
	:=\left(\frac{(\gamma-1)B_0-\frac{\gamma+1}{2}\lvert\nabla\varphi_0^+\rvert^2}{(\gamma-1)(B_0-\frac{1}{2}\lvert\nabla\varphi_0^+\rvert^2)^{1-\frac{1}{\gamma-1}}}\right)\nabla\chi\cdot{\bf n}_f.
\end{split}
\end{equation}
	According to Lemma \ref{Lem-S-bd-re} below, we can replace the boundary condition \eqref{shock-bd-h1} by the following condition:
	\begin{equation}\label{c2}
	(a_{ij}({\bf 0})\partial_i\chi)\cdot({-\bf n}_f)=g_1(\nabla\chi,f',{\bf t}_{\star})\mbox{ on }\mathfrak{S}_f
	\end{equation}
	with $g_1$ defined by 
	\begin{equation}\label{def-g1}
\begin{split}
	g_1&(\nabla\chi,f',{\bf t}_{\star})\\
		:=&\left(\frac{(\gamma-1)B_0-\frac{\gamma+1}{2}\lvert\nabla\varphi_0^+\rvert^2}{(\gamma-1)(B_0-\frac{1}{2}\lvert\nabla\varphi_0^+\rvert^2)^{1-\frac{1}{\gamma-1}}}\right)\left[\nabla\chi\cdot({\bf n}_f-{\bf b})+\mu_f\chi+\tilde{g}_1\right]\end{split}
\end{equation}
for ${\bf b}$, $\mu_f$, and $\tilde{g}_1$ specified in Lemma \ref{Lem-S-bd-re} below. 
	
\begin{lemma}\label{Lem-S-bd-re}
There exist a vector valued function ${\bf b}={\bf b}(\nabla\chi,f')$ and two functions $\mu_f$ and $\tilde{g}_1=\tilde{g}_1(\nabla\chi,f',{\bf t}_{\star})$ such that 
\begin{equation}\label{lem-re-bd}
\nabla\chi\cdot {\bf b}-\mu_f\chi=\tilde{g}_1\quad\mbox{on}\quad\mathfrak{S}_f.
\end{equation}
\end{lemma}
\begin{proof}
We prove this by adjusting \cite[Proof of Lemma 3.3]{bae2011transonic}.
First, it is clear that 
\begin{equation}\label{45-1}
\nabla\chi\cdot{\bf e}_r=\nabla\varphi\cdot{\bf e}_r-\nabla\varphi_0^+\cdot{\bf e}_r=\nabla\varphi\cdot{\bf e}_r-\frac{K_0}{\nabla\varphi_0^-\cdot{\bf e}_r}+\frac{K_0}{\nabla\varphi_0^-\cdot{\bf e}_r}-\nabla\varphi_0^+\cdot{\bf e}_r
\end{equation} 
for $K_0$ given in \eqref{K0-def}. 
Then, we have
\begin{equation}\label{45-chi-mu}
\nabla\chi\cdot{\bf e}_r-\mu_f\chi=h(\nabla\chi,\mu_f) \mbox{ on }\mathfrak{S}_f
\end{equation}
with 
\begin{equation*}
\begin{split}
&\mu_f:=\frac{\int_0^1\frac{d}{dr}\left(\frac{K_0}{\partial_r\varphi_0^-}-\partial_r\varphi_0^+\right)\left(r_{\rm sh}+t(f(\phi)-r_{\rm sh})\right)dt}{\int_0^1\partial_r(\varphi_0^--\varphi_0^+)(r_{\rm sh}+t(f(\phi)-r_{\rm sh}))dt},\\
&h(\nabla\chi,\mu_f):=\nabla(\chi+\varphi_0^+)\cdot{\bf e}_r-\frac{K_0}{\nabla\varphi_0^-\cdot{\bf e}_r}-\mu_f(\varphi^--\varphi_0^-).
\end{split}
\end{equation*}

For notational simplicity, let us set 
\begin{equation}\label{45-hast-1}
h_{\ast}(\nabla\chi):=\nabla(\chi+\varphi_0^+)\cdot{\bf e}_r-\frac{K_0}{\nabla\varphi_0^-\cdot{\bf e}_r}\quad\mbox{and}\quad \chi^-:=\varphi^--\varphi_0^-.
\end{equation}
We first decompose $h_{\ast}(\nabla\chi)$ in the form of $h_{\ast}(\nabla\chi)=\beta_{\star}\cdot\nabla\chi+h_{\star}$.
For that purpose, we rewrite $h_{\ast}(\nabla\chi)$ as
\begin{equation}\label{45-hast}
h_{\ast}(\nabla\chi)=\mathfrak{b}_1-\mathfrak{b}_2
\end{equation}
for $\mathfrak{b}_1$ and $\mathfrak{b}_2$ defined by 
\begin{equation*}
\begin{split}
&\mathfrak{b}_1:=\nabla(\chi+\varphi_0^+)\cdot({\bf e}_r-{\bf n}_f)-{\bf t}\cdot{\bf n}_f,\\
&\mathfrak{b}_2:=\frac{K_0}{\nabla\varphi_0^-\cdot{\bf e}_r}-\frac{K_s(f')}{\nabla\varphi^-\cdot{\bf n}_f+{\bf t}^-\cdot{\bf n}_f}
\end{split}
\end{equation*}
with $K_s(f')$ and ${\bf t}^-$ given in \eqref{def-KS} and \eqref{t--}, respectively.

Since ${\bf n}_f-{\bf e}_r$ can be written as 
\begin{equation}\label{ne-rel}
\begin{split}
{\bf n}_f-{\bf e}_r=&\,\nu(\nabla\chi^-,\nabla\chi)-\nu({\bf 0},{\bf 0})\\
=&\,j_1(\nabla\chi^-,\nabla\chi)\nabla\chi^-+j_2({\bf 0},\nabla\chi)\nabla\chi
\end{split}
\end{equation}
for $\nu$, $j_1$, and $j_2$  defined by 
\begin{equation*}
\begin{split}
&\nu(\xi,\eta):=\frac{(\nabla\varphi_0^-+\xi)-(\nabla\varphi_0^++\eta)}{|(\nabla\varphi_0^-+\xi)-(\nabla\varphi_0^++\eta)|},\\
&j_1(\xi,\eta):=\int_0^1 D_\xi\nu(t\xi,\eta)dt, \quad j_2(\xi,\eta):=\int_0^1D_\eta\nu(\xi, t\eta)dt\quad\mbox{for }\xi,\eta\in\mathbb{R}^3,
\end{split}
\end{equation*}
$\mathfrak{b}_1$ can be rewritten as 
\begin{equation}\label{45-b11}
\mathfrak{b}_1=\beta_1(\nabla\chi)\cdot\nabla\chi+h_1(\nabla\chi,{\bf t},f')
\end{equation}
for $\beta_1(\nabla\chi)$ and $h_1(\nabla\chi,{\bf t},f')$ defined by 
\begin{equation*}
\begin{split}
&\beta_1(\nabla\chi):=-[j_2({\bf 0},\nabla\chi)](\nabla\varphi_0^++\nabla\chi),\\
&h_1(\nabla\chi,{\bf t},f'):=-(\nabla\varphi_0^++\nabla\chi)\cdot [j_1(\nabla\chi^-,\nabla\chi)\nabla\chi^-]+{\bf t}\cdot{\bf n}_f.
\end{split}
\end{equation*}

Now we compute $\mathfrak{b}_2$. 
By the definitions of $K_0$ and $K_s(f')$, we have  
\begin{equation*}
K_0-K_s(f')=\frac{\gamma-1}{\gamma+1}\left(|\nabla\varphi_0^-\cdot{\bf e}_r|^2-|(\nabla\varphi^-+{\bf t}^-)\cdot{\bf n}_f|^2\right)+\frac{2\gamma}{\gamma+1}\left(\frac{p_0^-}{\rho_0^-}-\frac{p^-}{\rho^-}\right).
\end{equation*}
By using \eqref{ne-rel},  we get
\begin{equation*}
K_0-K_s(f')=\frac{\gamma-1}{\gamma+1}\left(\beta_2^{\ast}(\nabla\chi,f')\cdot\nabla\chi+h_2^{\ast}\right)+\frac{2\gamma}{\gamma+1}\left(\frac{p_0^-}{\rho_0^-}-\frac{p^-}{\rho^-}\right)
\end{equation*}
for $\beta_2^{\ast}(\nabla\chi,f')$ and $h_2^{\ast}(\nabla\chi,f')$ defined by 
\begin{equation*}
{\small
\begin{split}
\beta_2^{\ast}(\nabla\chi,f'):=&\,-[j_2({\bf 0},\nabla\chi)](\nabla\varphi_0^-\cdot{\bf e}_r+\nabla\varphi_0^-\cdot{\bf n}_f)\nabla\varphi_0^-,\\
h_2^{\ast}(\nabla\chi,f'):=&\,-(\nabla\varphi_0^-\cdot{\bf e}_r+\nabla\varphi_0^-\cdot{\bf n}_f)\nabla\varphi_0^-\cdot [j_1(\nabla\chi^-,\nabla\chi)\nabla\chi^-]\\
&\,+|\nabla\varphi_0^-\cdot{\bf n}_f|^2-|\nabla\varphi^-\cdot{\bf n}_f|^2-|{\bf t}^-\cdot{\bf n}_f|^2-2(\nabla\varphi^-\cdot{\bf n}_f)({\bf t}^-\cdot{\bf n}_f).
\end{split}}
\end{equation*}
Then $\mathfrak{b}_2$ can be rewritten as 
\begin{equation}\label{45-b22}
\mathfrak{b}_2=\beta_2(\nabla\chi,f')\cdot\nabla\chi+h_2(\nabla\chi,f')
\end{equation}
for $\beta_2(\nabla\chi,f')$ and $h_2(\nabla\chi,f')$ defined by 
\begin{equation*}
{\small
\begin{split}
\beta_2(\nabla\chi,f'):=&\,\frac{\gamma-1}{\gamma+1}\frac{\beta_2^{\ast}(\nabla\chi,f')}{\nabla\varphi_0^-\cdot{\bf e}_r}+\frac{K_s(f')[j_2({\bf 0},\nabla\chi)]{\bf e}_r}{(\nabla\varphi^-+{\bf t}^-)\cdot{\bf n}_f},\\
h_2(\nabla\chi,f'):=&\,\frac{\gamma-1}{\gamma+1}\frac{h_2^{\ast}(\nabla\chi,f')}{\nabla\varphi_0^-\cdot{\bf e}_r}+\frac{1}{\nabla\varphi_0^-\cdot{\bf e}_r}\frac{2\gamma}{\gamma+1}\left(\frac{p_0^-}{\rho_0^-}-\frac{p^-}{\rho^-}\right)\\
&\,+\frac{K_s(f')}{(\nabla\varphi^-+{\bf t}^-)\cdot{\bf n}_f}\left[\frac{\nabla\chi^-\cdot{\bf n}_r+{\bf t}^-\cdot{\bf n}_f}{\nabla\varphi_0^-\cdot{\bf e}_r}+{\bf e}_r\cdot[j_1(\nabla\chi^-,\nabla\chi)\nabla\chi^-]\right].
\end{split}}
\end{equation*}

Combining \eqref{45-hast}, \eqref{45-b11}, and \eqref{45-b22} gives 
\begin{equation}\label{hast-final}
h_{\ast}(\nabla\chi)=(\beta_1(\nabla\chi)-\beta_2(\nabla\chi,f'))\cdot\nabla\chi+(h_1(\nabla\chi,{\bf t},f')-h_2(\nabla\chi,f')).
\end{equation}
Then, it follows from \eqref{45-chi-mu}, \eqref{45-hast-1}, and \eqref{hast-final} that 
\begin{equation}
\nabla\chi\cdot{\bf b}-\mu_f\chi=\tilde{g}_1
\end{equation}
for 
\begin{equation*}
\begin{split}
&{\bf b}:={\bf e}_r-\beta_1(\nabla\chi)+\beta_2(\nabla\chi,f'),\\
&\tilde{g}_1:=h_1(\nabla\chi,{\bf t},f')-h_2(\nabla\chi,f')-\mu_f\chi^-.
\end{split}
\end{equation*}
The proof of Lemma \ref{Lem-S-bd-re} is completed.
\end{proof}

The boundary conditions in \eqref{chi-bd-w} and \eqref{chi-bd-ex} directly imply 
	\begin{equation}\label{c3}
	\begin{split}
	&(a_{ij}({\bf 0})\partial_i\chi)\cdot{\bf e}_{\phi}=g_2(\nabla\chi)\mbox{ on }\Gamma_{\rm w}\cap\partial\Om_f^+,\\
	&(a_{ij}({\bf 0})\partial_i\chi)\cdot{\bf e}_r=g_3(v_{\rm ex},\nabla\chi)\mbox{ on }\Gamma_{\rm ex}
	\end{split}
	\end{equation}
	for $g_2$ and $g_3$ defined by 
	\begin{equation*}
	\begin{split}
	&g_2(\nabla\chi):=-\left[\left(a_{ij}(\nabla\chi)-a_{ij}({\bf 0})\right)\partial_i\chi\right]\cdot{\bf e}_{\phi},\\	
	&g_3(v_{\rm ex},\nabla\chi):=v_{\rm ex}-v_c-\left[\left(a_{ij}(\nabla\chi)-a_{ij}({\bf 0})\right)\partial_i\chi\right]\cdot{\bf e}_r.
	\end{split}
	\end{equation*}
We finally collect the boundary conditions for $\chi$  from  \eqref{c1}, \eqref{c2}, and \eqref{c3} as follows:
\begin{equation}\label{chi-bd-1}
\left\{\begin{split}
&\chi+\varphi_0^+-\varphi^-=0,\quad(a_{ij}({\bf 0})\partial_i\chi)\cdot({-\bf n}_f)=g_1(\nabla\chi,f',{\bf t}_{\star})\mbox{ on }\mathfrak{S}_f,\\ 
&(a_{ij}({\bf 0})\partial_i\chi)\cdot{\bf e}_{\phi}=g_2(\nabla\chi)\mbox{ on }\Gamma_{\rm w}\cap\partial\Om_f^+,\\
&(a_{ij}({\bf 0})\partial_i\chi)\cdot{\bf e}_r=g_3(v_{\rm ex},\nabla\chi)\mbox{ on }\Gamma_{\rm ex}.
\end{split}\right.
\end{equation}


{\bf Step 2.} 
For a positive constant $\delta_3$  to be determined later, define an iteration set $\mathcal{I}(\delta_3)$  by 
\begin{equation}\label{set-delta4}
\mathcal{I}(\delta_3):=\left\{\chi\in C^{2,\alpha}_{(-1-\alpha,\Gamma_{\rm w})}({\Om^+})\,\middle\vert\,\begin{split}
							&\chi\mbox{ is axisymmetric},\\
							&\|\chi\|_{2,\alpha,\Om^+}^{(-1-\alpha,\Gamma_{\rm w})}\le \delta_3\sigma_v
							\end{split}\right\}.
\end{equation}
Fix $\chi_{\ast}\in\mathcal{I}(\delta_3)$ and $f\in \mathcal{B}_{\sigma_v}(r_{\rm sh})$. 
We consider the following linear boundary value problem for $\chi$ in a fixed domain $\Om_f^+$:
\begin{equation}\label{lin-chi-pro}
\left\{\begin{split}
\sum_{i,j=1}^3\partial_j(a_{ij}({\bf 0})\partial_i\chi)=\mbox{div}{\bf F}_{\ast}&\mbox{ in }\Om_f^+,\\
(a_{ij}({\bf 0})\partial_i\chi)\cdot({-\bf n}_f)=g_1&\mbox{ on }\mathfrak{S}_f,\\ 
(a_{ij}({\bf 0})\partial_i\chi)\cdot{\bf e}_{\phi}=g_2&\mbox{ on }\Gamma_{\rm w}\cap\partial\Om_f^+,\\
(a_{ij}({\bf 0})\partial_i\chi)\cdot{\bf e}_r=g_3&\mbox{ on }\Gamma_{\rm ex}
\end{split}\right.
\end{equation}
with ${\bf F}_{\ast}$ and $g_k$ $(k=1,2,3)$ defined  by 
\begin{equation*}
\begin{split}
&{\bf F}_{\ast}:={\bf F}(\nabla\chi_{\ast},{\bf t}_{\star}),\\
&g_1:=g_1(\nabla\chi_{\ast},f',{\bf t}_{\star}),\,g_2:=g_2(\nabla\chi_{\ast}),\,g_3:=g_3(v_{\rm ex},\nabla\chi_{\ast})
\end{split}
\end{equation*}
for ${\bf F}$ given in \eqref{Def-F}. 

\begin{lemma}\label{lemma-varphi-est}
There exists a constant $C_{\sharp}>0$ depending only on the data so that if 
\begin{equation}\label{sigma-ast-1}
(1+\delta_1+\delta_2+\delta_3)\sigma_v\le C_{\sharp},
\end{equation} 
 then we have the following estimates:
\begin{equation}\label{g123-est}
\begin{split}
&\|a_{ij}(\nabla\chi_{\ast})-a_{ij}({\bf 0})\|_{1,\alpha,\Om_f^+}^{(-\alpha,\Gamma_{\rm w})}\le C\delta_3\sigma_v,\\
&\|{\bf F}_{\ast}\|_{1,\alpha,\Om_f^+}^{(-\alpha,\Gamma_{\rm w})}\le C\left(\delta_1+\delta_2+\delta_3^2\sigma_v\right)\sigma_v,\\
&\begin{split}
\|g_1\|_{1,\alpha,\mathfrak{S}_f}^{(-\alpha,\partial\mathfrak{S}_f)}
		&\le C\left(1+\delta_1+\delta_2+\delta_3^2\sigma_v\right)\sigma_v,
	\end{split}\\
&\|g_2\|_{C^0(\overline{\Gamma_{\rm w}})}\le C\delta_3^2\sigma_v^2,\\
&\|g_3\|_{1,\alpha,\Gamma_{\rm ex}}^{(-\alpha,\partial\Gamma_{\rm ex})}\le C\sigma_v+C\delta_3^2\sigma_v^2
\end{split}
\end{equation}
for a constant $C>0$ depending only on the data. 
If $\delta_3\sigma_v$ is chosen sufficiently small depending only on the data, then the equation in \eqref{lin-chi-pro} is uniformly elliptic in $\Om_f^+$.
\end{lemma}

By a direct computation, one can prove the above lemma.
Then we get the following Lemma by similar ways  in \cite[Proof of Lemma 4.2]{park2020transonic} and \cite[Step 4 in Proof of Lemma 3.5]{bae2011transonic}.
\begin{lemma}\label{lemma-varphi}
For $\alpha\in(\frac{1}{2},1)$, the boundary value problem \eqref{lin-chi-pro} has a unique axisymmetric solution $\chi\in C^{2,\alpha}_{(-1-\alpha,\Gamma_{\rm w})}(\Om_f^+)$ that satisfies 
\begin{equation}\label{lin-chi-est}
\|\chi\|_{2,\alpha,\Om_f^+}^{(-1-\alpha,\Gamma_{\rm w})}\le C\left(\|{\bf F}_{\ast}\|_{1,\alpha,\Om_f^+}^{(-\alpha,\Gamma_{\rm w})}+\|g_1\|_{1,\alpha,\mathfrak{S}_f}^{(-\alpha,\partial\mathfrak{S}_f)}+\|g_2\|_{C^0(\overline{\Gamma_{\rm w}})}+\|g_3\|_{1,\alpha,\Gamma_{\rm ex}}^{(-\alpha,\partial\Gamma_{\rm ex})}\right)
\end{equation}
for a constant $C>0$ depending only on the data.
\end{lemma}


For the solution $\chi\in C^{2,\alpha}_{(-1-\alpha,\Gamma_{\rm w})}(\Om_f^+)$ of \eqref{lin-chi-pro} associated with $(f,\chi_{\ast})$, consider an extension $\chi_{\rm ext}$ defined in the same way as the definition \eqref{W-ext} of ${\bf W}_{\rm ext}$.
With such an extension, define an iteration mapping $\mathcal{J}_p:\mathcal{I}(\delta_3)\to C^{2,\alpha}_{(-1-\alpha,\Gamma_{\rm w})}({\Om^+})$ by 
\begin{equation*}
\mathcal{J}_p:\chi_{\ast}\mapsto\chi_{\rm ext}.
\end{equation*}
Then, by \eqref{g123-est}-\eqref{lin-chi-est} and the definition of the extension, $\chi_{\rm ext}$ satisfies the estimate
\begin{equation}\label{chi-ext-est}
\|\chi_{\rm ext}\|_{2,\alpha,\Om_f^+}^{(-1-\alpha,\Gamma_{\rm w})}\le C_{\natural}\left(1+\delta_1+\delta_2+\delta_3^2\sigma_v\right)\sigma_v
\end{equation}
for a constant $C_{\natural}>0$ depending only on the data.
If we choose $\delta_3$ and $\sigma_5$ satisfy 
\begin{equation}\label{sigma-ast-2}
\delta_3=2C_{\natural}(1+\delta_1+\delta_2)\quad\mbox{ and }\quad\sigma_5\le\frac{1}{2C_{\natural}\delta_3},
\end{equation}
then \eqref{chi-ext-est} implies that $\chi_{\rm ext}\in\mathcal{I}(\delta_3)$. The iteration mapping $\mathcal{J}_p$ maps $\mathcal{I}(\delta_3)$ into itself.

Since the iteration set $\mathcal{I}(\delta_3)$ given by \eqref{set-delta4} is a convex and compact subset of $C^{2,\alpha/2}_{(-1-\alpha/2,\Gamma_{\rm w})}({\Om^+})$ and 
$\mathcal{J}_p$ is continuous in $C^{2,\alpha/2}_{(-1-\alpha/2,\Gamma_{\rm w})}({\Om^+})$,
the Schauder fixed point theorem yields that there exists a fixed point $\chi\in\mathcal{I}(\delta_3)$ of $\mathcal{J}_p$.
Then, by using a contraction principle, one can prove that there exists a small $\sigma^{\diamond}>0$ depending on the data and $(\delta_1,\delta_2)$ so that if 
\begin{equation}\label{dia}
\sigma_5\le\sigma^{\diamond},
\end{equation}
 then the fixed point is unique.
For such a fixed point $\chi$, a function $\varphi$ defined by $\varphi:=\chi+\varphi_0^+$ is a solution to the following nonlinear boundary problem:
\begin{equation}\label{non-varphi}
\left\{\begin{split}
	&\mbox{div}\left({\bf A}(\nabla\varphi,{\bf t}_{\star})\right)=0\quad\mbox{in}\quad\Om_f^+,\\
	&-\nabla\varphi\cdot{\bf n}_f=-\frac{K_s(f')}{{\bf u}^-\cdot{\bf n}_f}+{\bf t}_{\star}\cdot{\bf n}_f\quad\mbox{on}\quad\mathfrak{S}_f,\\
	&\nabla\varphi\cdot{\bf e}_{\phi}=0\quad\mbox{on}\quad\Gamma_{\rm w}\cap\partial\Om_f^+,\\
	&\left(B_0-\frac{\lvert\nabla\varphi+{\bf t}_{\star}\rvert^2}{2}\right)^{\frac{1}{\gamma-1}}\nabla\varphi\cdot{\bf e}_r=v_{\rm ex}\quad\mbox{on}\quad\Gamma_{\rm ex}.
	\end{split}\right.
	\end{equation}

{\bf Step 3.} Due to \eqref{sigma-ast-1}, \eqref{sigma-ast-2}, and \eqref{dia}, we  need to consider $\sigma_5$ satisfying
\begin{equation}\label{sigma-dagger}
\sigma_5\le \min\left\{\sigma^{\diamond},\frac{C_{\sharp}}{1+\delta_1+\delta_2+\delta_3},\frac{1}{2C_{\natural}\delta_3},\frac{1}{4}\right\}=:\sigma_{\dagger}.
\end{equation}
For a positive constant $\sigma_{\dagger}$ above, define a mapping $\mathfrak{J}:\mathcal{B}_{\sigma_{\dagger}}(\rho_0^-,\varphi_0^-,p_0^-,v_c)\times\mathcal{B}_{\sigma_{\dagger}}(r_{\rm sh})\to C^{2,\alpha}_{(-1-\alpha,\{\phi=\phi_0\})}((0,\phi_0))$ by
\begin{equation}\label{ddef-J}
\mathfrak{J}:(\rho^-,\varphi^-,p^-,v_{\rm ex}, f)\mapsto (\varphi^--\varphi)(f(\phi),\phi)\mbox{ for }\phi\in(0,\phi_0),
\end{equation}
where $\varphi$ is the solution of the nonlinear boundary problem \eqref{non-varphi} associated with $(\rho^-,\varphi^-,p^-,v_{\rm ex}, f)$.
The mapping $\mathfrak{J}$ is well-defined by Lemma \ref{lemma-varphi}. 
One  can prove the following Lemma by slightly adjusting \cite[Proof of Lemma 3.11]{bae2011transonic}.
\begin{lemma}\label{Lemm-Impli} 
\begin{itemize}
\item[(i)] Obviously,  it holds that
$\mathfrak{J}(\rho_0^-,\varphi_0^-,p_0^-,v_{\rm c}, r_{\rm sh})=0.$
\item[(ii)] There exists $\sigma_{\diamond}\in(0,\sigma_\dagger]$ depending only on the data and $(\delta_1,\delta_2)$ so that 
 $\mathfrak{J}$ is continuously Fr\'echet differentiable in $\mathcal{B}_{\sigma_{\diamond}}(\rho_0^-,\varphi_0^-,p_0^-,v_{\rm c})\times \mathcal{B}_{\sigma_{\diamond}}(r_{\rm sh})$.
\item[(iii)] $D_f\mathfrak{J}(\rho_0^-,\varphi_0^-,p_0^-,v_{\rm c}, r_{\rm sh}):C^{2,\alpha}_{(-1-\alpha,\{\phi=\phi_0\})}((0,\phi_0))\to C^{2,\alpha}_{(-1-\alpha,\{\phi=\phi_0\})}((0,\phi_0))$ is invertible.
\end{itemize}
\end{lemma}

By Lemma \ref{Lemm-Impli} and the implicit function theorem, for $\sigma_{\diamond}$ given in Lemma \ref{Lemm-Impli}, there exists a constant $\sigma_5\in(0,\sigma_{\diamond}]$ depending only on the data and $(\delta_1,\delta_2)$ so that there is a unique mapping $\mathfrak{G}:\mathcal{B}_{\sigma_5}(\rho_0^-,\varphi_0^-,p_0^-,v_{\rm c})\to \mathcal{B}_{\sigma_5}(r_{\rm sh})$ satisfying 
\begin{equation*}
\mathfrak{J}({\bf p},\mathfrak{G}({\bf p}))=0\mbox{ for all }{\bf p}\in\mathcal{B}_{\sigma_5}(\rho_0^-,\varphi_0^-,p_0^-,v_{\rm c})
\end{equation*}
and $\mathfrak{G}$ is continuously Fr\'echet differentiable in $\mathcal{B}_{\sigma_5}(\rho_0^-,\varphi_0^-,p_0^-,v_{\rm c})$.
Then, for the solution $\varphi$ of the nonlinear boundary problem \eqref{non-varphi} associated with $(\rho^-,\varphi^-,p^-,v_{\rm ex}, f)=(\rho^-,\varphi^-,p^-,v_{\rm ex}, \mathfrak{G}(\rho^-,\varphi^-,p^-,v_{\rm ex}))$, $(\mathfrak{G}(\rho^-,\varphi^-,p^-,v_{\rm ex}),\varphi)$ is the unique solution of the free boundary problem \eqref{re-Lem-41}-\eqref{chi-bd-ex}.
The proof of Lemma \ref{Lem-f-var} is completed.
\qed



\section{Proof of Theorem \ref{Helmholtz-Theorem}}\label{Sec-Mthm}

\subsection{Proof of Proposition \ref{free-v-Pro}}\label{Section-Pro}
In \S\ref{Ss-pro} and \S\ref{Sec-Pro}, we proved Proposition \ref{free-v-Pro} (a).
To complete the proof of Proposition \ref{free-v-Pro}, it remains to prove Proposition \ref{free-v-Pro} (b), which is that the solution of the free boundary problem \eqref{free-v} is unique.

Let $(f^{(k)}, S^{(k)}, \Lambda^{(k)}, \varphi^{(k)}, {\bf W}^{(k)})$ ($k=1,2$) be two axisymmetric solutions of the free boundary problem \eqref{free-v}.
Suppose that each solution satisfies the estimate \eqref{free-v-est}.
Let ${\bf x}$ and ${\bf y}$ be the Cartesian coordinate systems for $\Om_{f^{(1)}}^+$ and $\Om_{f^{(2)}}^+$, respectively.
And, let $(r,\phi,\theta)$ and $(\tilde{r},\tilde{\phi},\tilde{\theta})$ be the spherical coordinate systems for $\Om_{f^{(1)}}^+$ and $\Om_{f^{(2)}}^+$, respectively.
For simplification, let us set $\mathfrak{T}$ as $\mathfrak{T}:=\mathfrak{T}_{f^{(1)},f^{(2)}}$.  And, let us set 
$\tilde{S}^{(2)}:=S^{(2)}\circ\mathfrak{T}$, $ \tilde{\Lambda}^{(2)}:=\Lambda^{(2)}\circ\mathfrak{T}$, $\tilde{\varphi}^{(2)}:=\varphi^{(2)}\circ\mathfrak{T}$, and $\tilde{\bf W}^{(2)}:={\bf W}^{(2)}\circ\mathfrak{T}$.
Then $\tilde{\bf W}^{(2)}$ satisfies 
\begin{equation}\label{EQ-W2}
\left\{\begin{split}
&-\Delta_{\bf x}\tilde{\bf W}^{(2)}=
\mbox{div}_{\bf x}{\bf H}+{\bf G}^{(2)}\,\,\mbox{in}\,\,\Om_{f^{(1)}}^+,\\
&-\nabla_{\bf x}\tilde{\bf W}^{(2)}\cdot{\bf n}_{f^{(1)}}+\mu(f^{(1)},(f^{(1)})')\tilde{\bf W}^{(2)}=\mathfrak{B}^{(2)}\,\,\mbox{on}\,\,\mathfrak{S}_{f^{(1)}},\\
&\tilde{\bf W}^{(2)}={\bf 0}\,\,\mbox{on}\,\,\partial\Om_{f^{(1)}}^+\backslash\mathfrak{S}_{f^{(1)}},
\end{split}\right.
\end{equation}
where ${\bf H}$, ${\bf G}^{(2)}$, and $\mathfrak{B}^{(2)}$ are defined as follows:

(i) For a jacobian matrix ${\bf M}:={\bf J}_{f^{(2)},f^{(1)}}=\left[\frac{\partial x_i}{\partial y_j}\right]_{i,j=1}^3$ and a $3\times3$ identity matrix $\mathbb{I}_3$, the function ${\bf H}$ is defined by
\begin{equation*}
{\bf H}:=\left(-\mathbb{I}_3+\frac{{\bf M}^T{\bf M}}{\det{\bf M}}\right)\nabla_{\bf x}\tilde{\bf W}^{(2)}.
\end{equation*}

(ii) For $G^{(2)}:=G\left(\tilde{r},\tilde{\phi},S^{(2)},\Lambda^{(2)},\partial_{\tilde{\phi}}S^{(2)},\partial_{\tilde{\phi}}\Lambda^{(2)},\nabla_{\bf y}\varphi^{(2)}+\nabla_{\bf y}\times{\bf W}^{(2)}+\frac{\Lambda^{(2)}}{\tilde{r}\sin\tilde{\phi}}{\bf e}_{\theta}\right)$ and $\tilde{G}^{(2)}:=G^{(2)}\circ\mathfrak{T}$,
the function ${\bf G}^{(2)}$ is defined by $${\bf G}^{(2)}:=\frac{1}{\det{\bf M}}\tilde{G}^{(2)}{\bf e}_{\theta}.$$

(iii) The function $\mathfrak{B}^{(2)}$ is defined by
\begin{equation*}
\begin{split}
\mathfrak{B}^{(2)}
:=&\,-\nabla_{\bf x}\tilde{\bf W}^{(2)}\cdot{\bf n}_{f^{(1)}}+\mu(f^{(1)},(f^{(1)})')\tilde{\bf W}^{(2)}\\
&-({\bf M}\cdot\nabla_{\bf x}\tilde{\bf W}^{(2)})\cdot{\bf n}_{f^{(2)}}+\mu(f^{(2)},(f^{(2)})')\tilde{\bf W}^{(2)}
-\mathcal{A}(f^{(2)},(f^{(2)})'){\bf e}_{\theta}
\end{split}
\end{equation*}
for $\mu$ and $\mathcal{A}$ given by \eqref{def-mu}.

Subtracting the boundary value problem \eqref{EQ-W2} for $\tilde{\bf W}^{(2)}$ from the boundary value problem for ${\bf W} ^{(1)}$ gives the following boundary value problem for $\tilde{\bf W}:={\bf W}^{(1)}-\tilde{\bf W}^{(2)}$ in $\Omega_{f^{(1)}}^+$:
\begin{equation}\label{tilde-W-eq}
\left\{\begin{split}
&\begin{split}
	-\Delta_{\bf x}\tilde{\bf W}
	=&\,\mbox{div}_{\bf x}{\bf H}+\left(G^{(1)}-\frac{\tilde{G}^{(2)}}{\det{\bf M}}\right){\bf e}_{\theta}=:\tilde{\bf G}\quad\mbox{in}\,\,\Om_{f^{(1)}}^+,
	\end{split}\\
&-\nabla_{\bf x}\tilde{\bf W}\cdot{\bf n}_{f^{(1)}}+\mu(f^{(1)},(f^{(1)})')\tilde{\bf W}=\mathcal{A}(f^{(1)},(f^{(1)})'){\bf e}_{\theta}-\mathfrak{B}^{(2)}=:\tilde{\mathcal{A}}\,\,\mbox{on}\,\,\mathfrak{S}_{f^{(1)}},\\
&\tilde{\bf W}={\bf 0}\,\,\mbox{on}\,\,\partial\Om_{f^{(1)}}^+\backslash\mathfrak{S}_{f^{(1)}}
\end{split}\right.
\end{equation}
for 
$G^{(1)}:=G\left(r,\phi,S^{(1)},\Lambda^{(1)},\partial_{\phi}S^{(1)},\partial_{\phi}\Lambda^{(1)},\nabla_{\bf x}\varphi^{(1)}+\nabla_{\bf x}\times{\bf W}^{(1)}+\frac{\Lambda^{(1)}}{r\sin\phi}{\bf e}_{\theta}\right).$
One can easily see that $\tilde{\bf W}_{nh}:=\tilde{\bf W}-\tilde{\bf W}_h $ with $\tilde{\bf W}_h$ satisfying 
\begin{equation*}
\left\{\begin{split}
&-\Delta_{\bf x}\tilde{\bf W}_h={\bf 0}\quad\mbox{in}\,\,D_{\eta}({\bf x}_0),\\
&-\nabla_{\bf x}\tilde{\bf W}_h\cdot{\bf n}_{f^{(1)}}+\mu(f^{(1)},(f^{(1)})')\tilde{\bf W}_h={\bf 0}\,\,\mbox{on}\,\,\partial D_{\eta}({\bf x}_0)\cap\mathfrak{S}_{f^{(1)}},\\
&\tilde{\bf W}_h={\bf 0}\,\,\mbox{on}\,\,\partial D_{\eta}({\bf x}_0)\cap\Gamma_{\rm w},\\
&\tilde{\bf W}_h=\tilde{\bf W}\,\,\mbox{on}\,\,\partial D_{\eta}({\bf x}_0)\cap\Omega_{f^{(1)}}^+,\\
\end{split}\right.
\end{equation*}
satisfies 
\begin{equation*}
\int_{D_{\eta}({\bf x}_0)}\sum_{i=1}^3\partial_i \tilde{\bf W}_{nh}\partial_i\xi d{\bf x}=-\int_{D_\eta({\bf x}_0)}\tilde{\bf G}\xi d{\bf x}
+\int_{\partial D_{\eta}({\bf x}_0)\cap\mathfrak{S}_f}\tilde{\mathcal{A}}\xi dS
\end{equation*}
for any $\xi\in\left\{\zeta\in H^1( D_{\eta}({\bf x}_0)):\zeta=0\mbox{ on }\partial D_{\eta}({\bf x}_0)\cap\left(\Gamma_{\rm w}\cup\Om_f^+\right)\right\}$.

By the definitions \eqref{def-H-G} of $\varrho$ and $G$, the functions $G^{(1)}$ and $\tilde{G}^{(2)}$ can be rewritten as follows:
\begin{equation*}
\begin{split}
&G^{(1)}=\left(\frac{\partial_{\phi}S^{(1)}}{S^{(1)}}\right)K^{(1)}+\frac{\Lambda^{(1)}(\partial_{\phi}\Lambda^{(1)})}{\sin^2\phi}L^{(1)},\\
&\tilde{G}^{(2)}=\left(\frac{J_{\phi}\cdot\nabla_{(r,\phi,\theta)}\tilde{S}^{(2)}}{\tilde{S}^{(2)}}\right){K}^{(2)}+\frac{\tilde{\Lambda}^{(2)}\left(J_{\phi}\cdot\nabla_{(r,\phi,\theta)}\tilde{\Lambda}^{(2)}\right)}{\sin^2\phi}L^{(2)},
\end{split}
\end{equation*}
where $J_\phi$, $K^{(k)}$ and $L^{(k)}$ ($k=1,2$) are defined by 
\begingroup
\begin{align*}\allowdisplaybreaks
&J_{\phi}:=\left(\frac{\partial {r}}{\partial \tilde{\phi}},\frac{\partial{\phi}}{\partial\tilde{\phi}},\frac{\partial{\theta}}{\partial\tilde{\phi}}\right),\\
&K^{(1)}:=\frac{1}{ r({\bf q}^{(1)}\cdot{\bf e}_r)\gamma}\left(B_0-\frac{1}{2}\lvert{\bf q}^{(1)}\rvert^2\right),\,\,
L^{(1)}:=\frac{1}{ r^3({\bf q}^{(1)}\cdot{\bf e}_r)},\\
&K^{(2)}:=\frac{1}{ \mathfrak{r}(\tilde{{\bf q}}^{(2)}\cdot{\bf e}_r)\gamma}\left(B_0-\frac{1}{2}\lvert\tilde{{\bf q}}^{(2)}\rvert^2\right),\,\,L^{(2)}:=\frac{1}{ \mathfrak{r}^3(\tilde{{\bf q}}^{(2)}\cdot{\bf e}_r)}
\end{align*}
\endgroup
for ${\bf q}^{(1)}:=\nabla\varphi^{(1)}+\nabla\times{\bf W}^{(1)}+\left(\frac{\Lambda^{(1)}}{r\sin\phi}\right){\bf e}_{\theta}$ and $\tilde{\bf q}^{(2)}:={\bf M}\cdot\nabla_{\bf x}\tilde{\varphi}^{(2)}+({\bf M}\cdot\nabla_{\bf x})\times\tilde{\bf W}^{(2)}+\left(\frac{\tilde{\Lambda}^{(2)}}{\mathfrak{r}\sin{\phi}}\right){\bf e}_{\theta}$ with 
$\mathfrak{r}:=\left[\frac{r_{\rm ex}-f^{(2)}(\phi)}{r_{\rm ex}-f^{(1)}(\phi)}\right](r-r_{\rm ex})+r_{\rm ex}(=\tilde{r})$.
Using the relations $\frac{\partial_{\phi}S}{S}=\partial_{\phi}(\ln S)$ and 
$\frac{\Lambda\partial_{\phi}\Lambda}{\sin^2\phi}=\frac{\Lambda}{\sin\phi}\left[\partial_{\phi}\left(\frac{\Lambda}{\sin\phi}\right)+\frac{\Lambda\cos\phi}{\sin^2\phi}\right]$,
we integrate by parts in the spherical coordinate space to 
deduce the estimate
\begin{equation}\label{Wtilde-est}
\int_{D_{\eta}({\bf x})}|\nabla\tilde{\bf W}|^2 d{\bf x}
\le C\eta^{1+2\alpha}\mathfrak{W}^2
\end{equation}
for
\begingroup
\begin{align*}\allowdisplaybreaks
\mathfrak{W}:=&\|S^{(1)}-\tilde{S}^{(2)}\|_{0,\alpha,{\Om_{f^{(1)}}^+}}^{(1-\alpha,\Gamma_{\rm w})}
+\sigma_v\mathfrak{W}_{\star},\\
\mathfrak{W}_{\star}:=&\|S^{(1)}-\tilde{S}^{(2)}\|_{0,\alpha,{\Om_{f^{(1)}}^+}}^{(1-\alpha,\Gamma_{\rm w})}+\left\|\frac{\Lambda^{(1)}}{\sin\phi}-\frac{\tilde{\Lambda}^{(2)}}{\sin\phi}\right\|_{0,\alpha,{\mathcal{R}_{f^{(1)}}^+}}^{(1-\alpha,\{\phi=\phi_0\})}\\
&+\|f^{(1)}-f^{(2)}\|_{1,\alpha,(0,\phi_0)}^{(-\alpha,\{\phi=\phi_0\})}
+\|\varphi^{(1)}-\tilde{\varphi}^{(2)}\|_{1,\alpha,{\Om_{f^{(1)}}^+}}^{(-\alpha,\Gamma_{\rm w})}+\|\tilde{\bf W}\|_{1,\alpha,{\Om_{f^{(1)}}^+}}^{(-\alpha,\Gamma_{\rm w})}.
\end{align*}
\endgroup
For a sufficiently small $\sigma_v$ depending only on the data, it holds that
\begin{equation}\label{Wtilde-1}
 \begin{split}
\|f^{(1)}-f^{(2)}\|_{1,\alpha,(0,\phi_0)}^{(-\alpha,\{\phi=\phi_0\})}+\|\varphi^{(1)}-\tilde{\varphi}^{(2)}\|_{1,\alpha,{\Om_{f^{(1)}}^+}}^{(-\alpha,\Gamma_{\rm w})}&\\
+\left\|\frac{\Lambda^{(1)}}{\sin\phi}-\frac{\tilde{\Lambda}^{(2)}}{\sin\phi}\right\|_{0,\alpha,{\mathcal{R}_{f^{(1)}}^+}}^{(1-\alpha,\{\phi=\phi_0\})}
&\le C\|\tilde{\bf W}\|_{1,\alpha,{\Om_{f^{(1)}}^+}}^{(-\alpha,\Gamma_{\rm w})}.
\end{split}
\end{equation}
Then, from this and \eqref{free-v-est}, we have
\begin{equation}\label{Wtilde-2}
\|S^{(1)}-\tilde{S}^{(2)}\|_{0,\alpha,{\Om_{f^{(1)}}^+}}^{(1-\alpha,\Gamma_{\rm w})}
\le C\sigma_v\|\tilde{\bf W}\|_{1,\alpha,{\Om_{f^{(1)}}^+}}^{(-\alpha,\Gamma_{\rm w})}.
\end{equation}
By the scaling argument with \eqref{Wtilde-est}-\eqref{Wtilde-2}, we have
\begin{equation}\label{W-cont}
\|\tilde{\bf W}\|_{1,\alpha,{\Om_{f^{(1)}}^+}}^{(-\alpha,\Gamma_{\rm w})}\le C^{\natural}\sigma_v \|\tilde{\bf W}\|_{1,\alpha,{\Om_{f^{(1)}}^+}}^{(-\alpha,\Gamma_{\rm w})}
\end{equation}
for a constant $C^{\natural}$ depending only on the data. 
Choose $\tilde{\sigma}_2^{\ast}$ satisfying
\begin{equation*}
0<\tilde{\sigma}_2^{\ast}\le\min\left\{\sigma_2^{\ast}, \frac{1}{2C^{\natural}}\right\}
\end{equation*}
for $\sigma_2^{\ast}$ in Proposition \ref{free-v-Pro} (a)
so that \eqref{W-cont} implies that $\tilde{\bf W}={\bf 0}$, i.e., ${\bf W}^{(1)}=\tilde{\bf W}^{(2)}$.
Then, by \eqref{Wtilde-1} and \eqref{Wtilde-2}, $(f^{(1)}, S^{(1)},\Lambda^{(1)},\varphi^{(1)},{\bf W}^{(1)})=(f^{(2)}, S^{(2)},\Lambda^{(2)},\varphi^{(2)},{\bf W}^{(2)})$.
This finishes the proof of Proposition \ref{free-v-Pro}.\qed

%

\subsection{Proof of Theorem \ref{Helmholtz-Theorem}}\label{proof-HThm}
We basically follow the idea of \cite{bae2011transonic} for potential flows.

\subsubsection{Existence: Proof of Theorem \ref{Helmholtz-Theorem} (a)}\label{S-exi-Thm}
For notational simplicity, set $\sigma:=\sigma_v$ and $\zeta_0:=(\rho_0^-,\varphi_0^-,p_0^-,v_{\rm c})$.
Define a mapping $\mathcal{P}:\mathcal{B}_{\sigma}(\zeta_0)\to\mathcal{B}_{C^{\star}\sigma}(p_{\rm c})$ by
\begin{equation}\label{def-mathcalP}
\left.\mathcal{P}:(\rho^-,\varphi^-,p^-,v_{\rm ex})\mapsto S\varrho^{\gamma}\left(S,\nabla\varphi+\nabla\times{\bf W}+\left(\frac{\Lambda}{r\sin\phi}\right){\bf e}_{\theta}\right)\right\vert_{\Gamma_{\rm ex}},
\end{equation}
where $(f, S, \Lambda, \varphi, {\bf W})$ is the solution to the free boundary problem  \eqref{free-v} associated with $(\rho^-,\varphi^-,p^-,v_{\rm ex})$.
By \eqref{free-v-est}, one can choose a suitable constant $C^{\star}$ so that the mapping $\mathcal{P}$ is well-defined.

To prove the local invertibility of $\mathcal{P}$, we need the following lemma.

\begin{lemma}\label{P-empty}
\begin{itemize}
\item[(i)] $\mathcal{P}$ is Fr\'echet differentiable at $\zeta_0$;
\item[(ii)] $\mathcal{P}:\mathcal{B}_{\sigma}(\zeta_0)\to\mathcal{B}_{C^{\star}\sigma}(p_{\rm c})$ is continuous in the sense that if $\zeta_j:=(\rho_j^-,\varphi_j^-,p_j^-,v_j)$ converges to $\zeta_{\infty}:=(\rho_\infty^-,\varphi_\infty^-,p_\infty^-,v_\infty)$ in $\mathfrak{B}^{(1)}:=C^{2,\frac{\alpha}{2}}(\Omega)\times C^{3,\frac{\alpha}{2}}(\Omega)\times C^{2,\frac{\alpha}{2}}(\Omega)\times C^{1,\frac{\alpha}{2}}_{(-\frac{\alpha}{2},\partial\Gamma_{\rm ex})}(\Gamma_{\rm ex})$ as $j$ tends to $\infty$, then $\mathcal{P}(\zeta_j)$ converges to $\mathcal{P}(\zeta_{\infty})$ in $\mathfrak{B}^{(2)}:=C^{1,\frac{\alpha}{2}}_{(-\frac{\alpha}{2},\partial\Gamma_{\rm ex})}(\Gamma_{\rm ex})$;
\item[(iii)] the partial Fr\'echet derivative  of $\mathcal{P}$ at $\zeta_0$ with respect to $v_{\rm ex}$ is invertible.
\end{itemize}
\end{lemma}
Once Lemma \ref{P-empty} is verified, the weak implicit mapping theorem (cf. \cite[Proposition 5.1]{bae2011transonic}) implies that 
$\mathcal{P}^{-1}(p_{\rm ex})\ne\emptyset$ for any $p_{\rm ex}\in\mathcal{B}_{\tilde{\sigma}}(p_c)$ for a sufficiently small $\tilde{\sigma}>0$.
Moreover, according to Remark A.2 in the proof of \cite[Proposition 5.1]{bae2011transonic}, we have
\begin{equation}\label{vex-pex}
\|v_{\rm ex}-v_{\rm c}\|_{1,\alpha,\Gamma_{\rm ex}}^{(-\alpha,\partial\Gamma_{\rm ex})}\le C\sigma(\rho^-,{\bf u}^-,p^-,p_{\rm ex})
\end{equation}
as desired in \eqref{v-vc}.
Then, we finally choose $\sigma_2$ satisfying $0<\sigma_2\le\min\{\frac{\tilde{\sigma}}{C^{\star}},\tilde{\sigma}_2^{\ast}\}$ for $\tilde{\sigma}_2^{\ast}$ given in Proposition \ref{free-v-Pro} to complete the proof of Theorem \ref{Helmholtz-Theorem} (a).
The rest of this subsection is devoted to proving Lemma \ref{P-empty}.


\begin{proof}[Proof of Lemma \ref{P-empty} (i)]
We only compute the partial derivative of $\mathcal{P}$ with respect to $v_{\rm ex}$ at $\zeta_0$, because other partial derivatives of $\mathcal{P}$ can be obtained similarly. 

Note that 
\begin{equation*}
p=S\rho^{\gamma}=S\left(\frac{\gamma-1}{\gamma S}\right)^{\frac{\gamma}{\gamma-1}}\left(B_0-\frac{1}{2}|\nabla\varphi+{\bf t}|^2\right)^{\frac{\gamma}{\gamma-1}}.
\end{equation*}
For $\mathcal{P}$ given in \eqref{def-mathcalP}, define $\mathcal{P}^{(1)}$  and $\mathcal{P}^{(2)}$ by 
\begin{equation}\label{def-RQ}
\begin{split}
&\mathcal{P}^{(1)}:(\rho^-,\varphi^-,p^-,v_{\rm ex})\mapsto\left.\left(B_0-\frac{1}{2}\lvert\nabla\varphi+{\bf t}\rvert^2\right)^{\frac{\gamma}{\gamma-1}}\right\rvert_{\Gamma_{\rm ex}},\\
&\mathcal{P}^{(2)}:(\rho^-,\varphi^-,p^-,v_{\rm ex})\mapsto\frac{\mathcal{P}(\rho^-,\varphi^-,p^-,v_{\rm ex})}{\mathcal{P}^{(1)}(\rho^-,\varphi^-,p^-,v_{\rm ex})}=\left.S\left(\frac{\gamma-1}{\gamma S}\right)^{\frac{\gamma}{\gamma-1}}\right|_{\Gamma_{\rm ex}}.
\end{split}
\end{equation}
Then the proof of Lemma \ref{P-empty} (i) directly follows  from Lemma \ref{lemma-RQ} below.
\end{proof}

\begin{lemma}\label{lemma-RQ}{\rm (Analogy of \cite[Lemma 5.2]{bae2011transonic})} 
The mappings $\mathcal{P}^{(1)}$, $\mathcal{P}^{(2)}$, and $\mathcal{P}$ are Fr\'echet differentiable at $\zeta_0$, and 
the partial Fr\'echet derivatives of $\mathcal{P}^{(1)}$, $\mathcal{P}^{(2)}$, and $\mathcal{P}$ at $\zeta_0$ with respect to $v_{\rm ex}$ are given by 
\begin{eqnarray}
&\label{DR}&D_{v}\mathcal{P}^{(1)}:w\mapsto \left.-\frac{\gamma\left(B_0-\frac{1}{2}|\nabla\varphi_0^+|^2\right)^{\frac{1}{\gamma-1}}}{\gamma-1}\nabla\varphi_0^+\cdot\nabla\varphi^{(w)}\right|_{\Gamma_{\rm ex}},\\
&\label{DQ}&D_{v}\mathcal{P}^{(2)}:w\mapsto \left.-\frac{1}{\gamma-1}\left(\frac{\gamma-1}{\gamma }\right)^{\frac{\gamma}{\gamma-1}}(S_0^+)^{\frac{-\gamma}{\gamma-1}}{S}^{(w)}\right|_{\Gamma_{\rm ex}},\\
&\label{DP}&D_{v}\mathcal{P}:w\mapsto \mathcal{P}^{(2)}(\zeta_0)D_{v}\mathcal{P}^{(1)}(w)+\mathcal{P}^{(1)}(\zeta_0)D_{v}\mathcal{P}^{(2)}(w)
\end{eqnarray}
for any $w\in C^{1,\alpha}_{(-\alpha,\partial\Gamma_{\rm ex})}(\Gamma_{\rm ex})$.
In \eqref{DR}-\eqref{DQ}, $\varphi^{(w)}$ and $S^{(w)}$ are specified in the proof below.
\end{lemma}
\begin{proof}
The proof is divided into three steps. 

{\bf 1.} 
Fix a function $w\in C^{1,\alpha}_{(-\alpha,\partial\Gamma_{\rm ex})}(\Gamma_{\rm ex})$ with $\|w\|_{1,\alpha,\Gamma_{\rm ex}}^{(-\alpha,\partial\Gamma_{\rm ex})}=1$ and fix a sufficiently small positive constant $\epsilon_0\in(0,\tilde{\sigma}_2^{\ast}]$.
According to Proposition \ref{free-v-Pro}, for $\epsilon\in[-\epsilon_0,\epsilon_0]\backslash\{0\}$, there exists a unique solution $(f_{\epsilon},S_{\epsilon},\Lambda_{\epsilon},\varphi_{\epsilon},{\bf W}_{\epsilon})$  of \eqref{free-v} with replacing $({\bf u}^-,\rho^-,p^-,v_{\rm ex})$ by $(\nabla\varphi_0^-,\rho_0^-,p_0^-,v_{\rm c}+\epsilon w)$. 
And the solution satisfies the estimate \eqref{free-v-est}.
Note that, since $({\bf u}^-,\rho^-,p^-)=(\nabla\varphi_0^-,\rho_0^-,p_0^-)$, we have
 $\Lambda_{\epsilon}\equiv0$ and $(f_0,S_0,\Lambda_0,\varphi_0,{\bf W}_0)=(r_{\rm sh},S_0^+,0,\varphi_0^+,{\bf 0})$.
For ${\bf t}_{\epsilon}:=\nabla\times{\bf W}_{\epsilon}$, $(f_{\epsilon},S_{\epsilon},\varphi_{\epsilon},{\bf W}_{\epsilon})$ satisfy
\begingroup
\begin{align*}\allowdisplaybreaks
&\left\{\begin{aligned}
	&\mbox{div}\left({\bf A}(\nabla\varphi_{\epsilon},{\bf t}_{\epsilon})\right)=0\\
	&{\bf A}(\nabla\varphi_{\epsilon},{\bf t}_{\epsilon})\cdot\nabla S_{\epsilon}=0\\
	&-\Delta{\bf W}_{\epsilon}=G(r,\phi,S_{\epsilon},0,\partial_{\phi}S_{\epsilon},0,\nabla\varphi_{\epsilon}+{\bf t}_{\epsilon}){\bf e}_{\theta}
	\end{aligned}\right.\quad\mbox{in}\quad\Om_{f_{\epsilon}}^+,\\
&\left\{\begin{aligned}
		&\varphi_{\epsilon}=\varphi^-,\quad-\nabla\varphi_{\epsilon}\cdot{\bf n}_{f_{\epsilon}}=-\frac{K_s(f_{\epsilon}')}{{\bf u}^-\cdot{\bf n}_{f_{\epsilon}}}+{\bf t}_{\epsilon}\cdot{\bf n}_{f_{\epsilon}}\quad\mbox{on}\quad\mathfrak{S}_{f_{\epsilon}},\\
	&\nabla\varphi_{\epsilon}\cdot{\bf e}_{\phi}=0\quad\mbox{on}\quad\Gamma_{\rm w}\cap\partial\Om_{f_{\epsilon}}^+,\\
	&\left(B_0-\frac{\lvert\nabla\varphi_{\epsilon}+{\bf t}_{\epsilon}\rvert^2}{2}\right)^{\frac{1}{\gamma-1}}\nabla\varphi_{\epsilon}\cdot{\bf e}_r=v_{\rm c}+\epsilon w\quad\mbox{on}\quad\Gamma_{\rm ex},\\
	&S_{\epsilon}=S_{\rm sh}(f_{\epsilon}')\quad\mbox{on}\quad\mathfrak{S}_{f_{\epsilon}},\\
	&-\nabla{\bf W}_{\epsilon}\cdot{\bf n}_{f_{\epsilon}}+\mu(f_{\epsilon},f_{\epsilon}'){\bf W}_{\epsilon}=\mathcal{A}(f_{\epsilon},f_{\epsilon}'){\bf e}_{\theta}={\bf 0}\,\mbox{ on }\,\mathfrak{S}_{f_{\epsilon}},\\
	&{\bf W}_{\epsilon}={\bf 0}\,\mbox{ on }\,\partial\Om_{f_{\epsilon}}^+\backslash\mathfrak{S}_{f_{\epsilon}},
	\end{aligned}\right.
\end{align*}
\endgroup
where ${\bf A}$, $G$, $K_s(f')$, $v_c$, $S_{\rm sh}(f')$, $\mu(f,f')$, and $\mathcal{A}(f,f')$ are defined in \eqref{def-H-G}, \eqref{def-KS}, \eqref{v-c-def}, \eqref{def-A}, and \eqref{def-mu}.

Set $(\tilde{S}_{\epsilon},\tilde{\varphi}_{\epsilon},\tilde{\bf W}_{\epsilon}):=(S_{\epsilon},\varphi_{\epsilon},{\bf W}_{\epsilon})\circ\mathfrak{T}_{r_{\rm sh},f_{\epsilon}}$ and compute the G\^{a}teaux derivatives 
\begin{equation*}
\begin{split}
&f^{(w)}:=\lim_{\epsilon\to0}\frac{f_{\epsilon}-f_0}{\epsilon},\quad 
{S}^{(w)}:=\lim_{\epsilon\to0}\frac{\tilde{S}_{\epsilon}-S_0}{\epsilon},\\
&{\varphi}^{(w)}:=\lim_{\epsilon\to0}\frac{\tilde{\varphi}_{\epsilon}-\varphi_0}{\epsilon},\quad
{\bf W}^{(w)}:=\lim_{\epsilon\to0}\frac{\tilde{\bf W}_{\epsilon}-{\bf W}_0}{\epsilon}.
\end{split}
\end{equation*}
%
Since $(\varphi_{\epsilon}-\varphi^-)(f_{\epsilon}(\phi),\phi)\equiv 0$ and $f_0=r_{\rm sh}$, we have
\begin{equation*}
\begin{split}
\varphi^{(w)}(r_{\rm sh},\phi)
&=\lim_{\epsilon\to0}\frac{(\varphi_{\epsilon}-\varphi^-)+(\varphi_0^--\varphi_0^+)}{\epsilon}(f_{\epsilon}(\phi),\phi)\\
&=\partial_r(\varphi_0^--\varphi_0^+)(r_{\rm sh})f^{(w)},
\end{split}
\end{equation*}
from which we obtain that 
\begin{equation}\label{f-FD}
f^{(w)}=\frac{{\varphi}^{(w)}(r_{\rm sh},\phi)}{\partial_r(\varphi_0^--\varphi_0^+)(r_{\rm sh})}.
\end{equation}
%
%
Since $S_0=S_0^+$ and $S_{\epsilon}$ satisfies 
\begin{equation*}\label{S-epsilon-eq}
\left\{\begin{split}
&{\bf A}(\nabla\varphi_{\epsilon},\nabla\times{\bf W}_{\epsilon})\cdot\nabla S_{\epsilon}=0\quad\mbox{in}\quad\Om_{f_{\epsilon}}^+,\\
 &S_{\epsilon}=S_{\rm sh}(f_{\epsilon}')\quad\mbox{on}\quad \mathfrak{S}_{f_{\epsilon}},
 \end{split}\right.
\end{equation*}
 one can see that $\tilde{S}_{\epsilon}$ satisfies
\begin{equation*}
\left\{\begin{split}
&{\mathcal{A}}_{\tilde{r}}\partial_{\tilde{r}}\tilde{S}_{\epsilon}+{\mathcal{A}_{\phi}}\partial_{\phi}\tilde{S}_{\epsilon}=0\quad\mbox{in}\quad\mathcal{R},\\
 &\tilde{S}_{\epsilon}=S_{\rm sh}(f_{\epsilon}')\circ\mathfrak{T}_{r_{\rm sh},f_{\epsilon}}\quad\mbox{on}\quad \partial\mathcal{R}\cap\{\tilde{r}=r_{\rm sh}\},
 \end{split}\right.
\end{equation*}
where $\mathcal{A}_{\tilde{r}}$ and $\mathcal{A}_{\phi}$ are defined by 
\begin{equation*}
\begin{split}
&\mathcal{A}_{\tilde{r}}:=(r_{\rm ex}-f_{\epsilon}(\phi))a^2_{f_{\epsilon}}\sin\phi({\bf M}_{\epsilon}\cdot{\bf e}_r)b_{f_{\epsilon}}+(r_{\rm ex}-f_{\epsilon}(\phi))a_{f_{\epsilon}}\sin\phi({\bf M}_{\epsilon}\cdot{\bf e}_{\phi})c_{f_{\epsilon}},\\
&\mathcal{A}_{\phi}:=(r_{\rm ex}-f_{\epsilon}(\phi))a_{f_{\epsilon}}\sin\phi({\bf M}_{\epsilon}\cdot{\bf e}_{\phi})
\end{split}
\end{equation*}
for
\begin{equation*}
\begin{split}
&a_{f_{\epsilon}}:=\left(\frac{r_{\rm ex}-f_{\epsilon}(\phi)}{r_{\rm ex}-r_{\rm sh}}\right)(\tilde{r}-r_{\rm ex})+r_{\rm ex},\\
&{\bf M}_{\epsilon}:={\bf A}(\nabla\varphi_{\epsilon},\nabla\times{\bf W}_{\epsilon}+\left(\frac{\Lambda_{\ast}}{r\sin\phi}\right){\bf e}_{\theta})\circ\mathfrak{T}_{r_{\rm sh},f_{\epsilon}},\\
&b_{f_{\epsilon}}:=\frac{r_{\rm ex}-r_{\rm sh}}{r_{\rm ex}-f_{\epsilon}(\phi)},\quad
c_{f_{\epsilon}}:=\left(\frac{\tilde{r}-r_{\rm ex}}{r_{\rm ex}-f_{\epsilon}(\phi)}\right)f_{\epsilon}'(\phi).
\end{split}
\end{equation*}
Then, one can see that $\tilde{S}_{\epsilon}$ is defined by 
\begin{equation*}
\tilde{S}_{\epsilon}:=\left(S_{\rm sh}(f_{\epsilon}')\circ\mathfrak{T}_{r_{\rm sh},f_{\epsilon}}\right)\circ\mathcal{K}_{\epsilon}
\end{equation*}
for a function $\mathcal{K}_{\epsilon}:\mathcal{R}\to[0,\phi_0]$ defined by 
\begin{equation*}
\mathcal{K}_{\epsilon}:=\mathcal{G}_{\epsilon}^{-1}\circ w_{\epsilon}
\end{equation*}
with an invertible function $\mathcal{G}_{\epsilon}:[0,\phi_0]\to[w_{\epsilon}(r_{\rm sh},0),w_{\epsilon}(r_{\rm sh},\phi_0)]$ and a function $w_{\epsilon}:\mathcal{R}\to[w_{\epsilon}(r_{\rm sh},0),w_{\epsilon}(r_{\rm sh},\phi_0)]$ defined by 
\begin{equation*}
\begin{split}
&w_{\epsilon}(\tilde{r},\phi):=\int_0^{\phi}{\mathcal{A}}_{\tilde{r}}(f_{\epsilon},\nabla\varphi_{\epsilon},\nabla\times{\bf W}_{\epsilon})(\tilde{r},z)dz\mbox{ for }(\tilde{r},\phi)\in\overline{\mathcal{R}},\\
&\mathcal{G}_{\epsilon}(\phi):=w_{\epsilon}(r_{\rm sh},\phi)\mbox{ for }\phi\in[0,\phi_0].
\end{split}
\end{equation*}
It follows from \eqref{def-A} that
\begin{equation}\label{Sep-S0}
\begin{split}
{\tilde{S}_{\epsilon}-S_0}
&={\left(S_{\rm sh}(f_{\epsilon}')\circ\mathfrak{T}_{r_{\rm sh},f_{\epsilon}}\right)\circ\mathcal{K}_{\epsilon}
-\left(S_{\rm sh}(f_{0}')\circ\mathfrak{T}_{r_{\rm sh},f_{0}}\right)\circ\mathcal{K}_{0}}\\
&={A_1(f_{\epsilon}(\phi),\mathcal{K}_{\epsilon}(\tilde{r},\phi))-A_0(r_{\rm sh},\phi)},
\end{split}
\end{equation}
where we set $A_1:=A(\rho_0^-,(\partial_r\varphi_0^-)^2,p_0^-,f_{\epsilon}^2|f_{\epsilon}'|^2)$ and $A_0:=A(\rho_0^-,(\partial_r\varphi_0^-)^2,p_0^-,0)$ for a function $A$ defined by
\begin{equation*}
A(z_1,z_2,z_3,z_4):=[z_1z_2+z_3-z_1K(z_1,z_2,z_3,z_4)]\left[\frac{z_1z_2}{K(z_1,z_2,z_3,z_4)}\right]^{-\gamma}
\end{equation*}
with 
\begin{equation*}
K(z_1,z_2,z_3,z_4):=\frac{2(\gamma-1)}{\gamma+1}\left(\frac{1}{2}\cdot\frac{z_2}{1+z_4}+\frac{\gamma z_3}{(\gamma-1)z_1}\right).
\end{equation*}
%
Dividing \eqref{Sep-S0} by $\epsilon$ and taking $\epsilon\to0$
yield that
\begin{equation}\label{S-FD}
{S}^{(w)}=\lim_{\epsilon\to0}\frac{\tilde{S}_{\epsilon}-S_0}{\epsilon}
=A_{\ast}f^{(w)} 
\end{equation}
for a constant $A_{\ast}$ defined by 
\begin{equation}\label{def-Astar}
\begin{split}
A_{\ast}:=\partial_rA_0(r_{\rm sh},\phi)=&\,\partial_{z_1}A(\rho_0^-,(\partial_r\varphi_0^-)^2,p_0^-,0)(\partial_r\rho_0^-)\\
&+\partial_{z_2}A(\rho_0^-,(\partial_r\varphi_0^-)^2,p_0^-,0)2(\partial_r\varphi_0^-)(\partial_{rr}\varphi_0^-)\\
&+\partial_{z_3}A(\rho_0^-,(\partial_r\varphi_0^-)^2,p_0^-,0)(\partial_rp_0^-).
\end{split}
\end{equation}
By a direct computation and properties of  $(\rho_0^-,\partial_r\varphi_0^-,p_0^-,\partial_{rr}\varphi_0^-)$, one can check that $A_{\ast}>0$. The sign of $A_{\ast}$ is important for proving Lemma \ref{P-empty} (iii).

Similar to \eqref{tilde-W-eq}, we can get equations for $\tilde{\varphi}_{\epsilon}-\varphi_0$ and $\tilde{\bf W}_{\epsilon}-{\bf W}_0$. Dividing the equations by $\epsilon$ and formally taking $\epsilon\to0$ give the following system:
\begin{equation}\label{sys-var}
\left\{\begin{split}
\sum_{i,j=1}^3\partial_j(a_{ij}({\bf 0})\partial_i{\varphi}^{(w)})=\mbox{div}{\mathfrak{H}}&\mbox{ in }\Om_{r_{\rm sh}}^+,\\
(a_{ij}({\bf 0})\partial_i{\varphi}^{(w)})\cdot{\bf e}_r=\mathfrak{g}&\mbox{ on }\mathfrak{S}_{r_{\rm sh}},\\ 
(a_{ij}({\bf 0})\partial_i{\varphi}^{(w)})\cdot{\bf e}_{\phi}=0&\mbox{ on }\Gamma_{\rm w}\cap\partial\Om_{r_{\rm sh}}^+,\\
(a_{ij}({\bf 0})\partial_i{\varphi}^{(w)})\cdot{\bf e}_r=w&\mbox{ on }\Gamma_{\rm ex}
\end{split}\right.
\end{equation}
for $\mathfrak{H}=(\mathfrak{h}_1,\mathfrak{h}_2,\mathfrak{h}_3)$ and $\mathfrak{g}$ defined by 
\begin{equation}\label{tvex}
\begin{split}
&\mathfrak{h}_i:=-\left(B_0^+-\frac{|\nabla\varphi_0^+|^2}{2}\right)^{\frac{1}{\gamma-1}}\left(\nabla\times {\bf W}^{(w)}\right)\cdot{\bf e}_i,\quad i=1,2,3,\\
&\begin{split}\mathfrak{g}:=&\frac{(\gamma-1)B_0-\frac{\gamma+1}{2}\lvert\nabla\varphi_0^+\rvert^2}{(\gamma-1)(B_0-\frac{1}{2}\lvert\nabla\varphi_0^+\rvert^2)^{1-\frac{1}{\gamma-1}}}
		\left(\mu_0\varphi^{(w)}+(\nabla\times{\bf W}^{(w)})\cdot{\bf e}_r\right),
		\end{split}
\end{split}
\end{equation}
and 
\begin{equation}\label{sys-W}
\left\{\begin{split}
-\Delta{\bf W}^{(w)}=\frac{1}{r(\nabla\varphi_0^+\cdot{\bf e}_r)}\frac{\partial_{\phi}S^{(w)}}{\gamma S_0^+}\left(B_0^+-\frac{|\nabla\varphi_0^+|^2}{2}\right)&\mbox{ in }\Om_{r_{\rm sh}}^+,\\
-\nabla{\bf W}^{(w)}\cdot{\bf e}_r-\frac{1}{r^2}{\bf W}^{(w)}={\bf0}&\mbox{ on }\mathfrak{S}_{r_{\rm sh}},\\
{\bf W}^{(w)}={\bf0}&\mbox{ on }\partial\Om_{r_{\rm sh}}^+\backslash\mathfrak{S}_{r_{\rm sh}}.
\end{split}\right.
\end{equation}
As in Lemmas \ref{lemma-varphi} and \ref{Pro2},
we can obtain a unique solution $(\varphi^{(w)},{\bf W}^{(w)})$ to the system \eqref{sys-var}-\eqref{sys-W}.
Then, a direct computation for $\tilde{S}_{\epsilon}-S_0-\epsilon{S}^{(w)}$ with \eqref{Sep-S0}-\eqref{S-FD} and Lemmas \ref{lemma-varphi} and \ref{Pro2} yield that 
$(f^{(w)},S^{(w)},\varphi^{(w)},{\bf W}^{(w)})$ satisfy
\begin{equation}\label{FD-est}
\begin{split}
&\|f_{\epsilon}-f_0-\epsilon f^{(w)}\|_{2,\alpha,(0,\phi_0)}^{(-1-\alpha,\{\phi=\phi_0\})}\le C\epsilon^2,\\
&\|\tilde{S}_{\epsilon}-S_0-\epsilon{S}^{(w)}\|_{1,\alpha,\Om^+_{r_{\rm sh}}}^{(-\alpha,\Gamma_{\rm w})}\le C\epsilon^2,\\
&\|\tilde{\varphi}_{\epsilon}-\varphi_0-\epsilon{\varphi}^{(w)}\|_{2,\alpha,\Om^+_{r_{\rm sh}}}^{(-1-\alpha,\Gamma_{\rm w})}\le C\epsilon^2,\\
&\|\tilde{\bf W}_{\epsilon}-{\bf W}_0-\epsilon{\bf W}^{(w)}\|_{2,\alpha,\Om^+_{r_{\rm sh}}}^{(-1-\alpha,\Gamma_{\rm w})}\le C\epsilon^2.
\end{split}
\end{equation}


{\bf 2.} (Partial Fr\'echet derivative of $\mathcal{P}^{(1)}$)
By the definition of $\mathcal{P}^{(1)}$, we have 
\begin{equation}\label{diff-P1}
\mathcal{P}^{(1)}(\rho_0^-,\varphi_0^-,p_0^-,v_{\rm c}+\epsilon w)=\left.\left(B_0-\frac{1}{2}\lvert\nabla\varphi_{\epsilon}+{\bf t}_{\epsilon}\rvert^2\right)^{\frac{\gamma}{\gamma-1}}\right\rvert_{\Gamma_{\rm ex}}.
\end{equation}
Since ${\bf W}_{\epsilon}={\bf 0}$ and $\Lambda_{\epsilon}\equiv0$  on $\Gamma_{\rm ex}$,  it holds that  ${\bf t}_{\epsilon}=-\partial_r({\bf W}_{\epsilon}\cdot{\bf e}_{\theta}){\bf e}_{\phi}\mbox{ on }\Gamma_{\rm ex}$. 
Using this, we get
\begin{equation*}
\begin{split}
&\left.\frac{d}{d\epsilon}\right|_{\epsilon=0}\mathcal{P}^{(1)}(\rho_0^-,\varphi_0^-,p_0^-,v_{\rm c}+\epsilon w)
=\left.-\frac{\gamma}{\gamma-1}\left(B_0-\frac{1}{2}\lvert\nabla\varphi_{0}\rvert^2\right)^{\frac{1}{\gamma-1}}\nabla\varphi_0\cdot\nabla\varphi^{(w)}\right|_{\Gamma_{\rm ex}}.
\end{split}
\end{equation*}
Then, it follows from \eqref{FD-est} that 
\begin{equation*}
\|\mathcal{P}^{(1)}(\rho_0^-,\varphi_0^-,p_0^-,v_{\rm c}+\epsilon w)-\mathcal{P}^{(1)}(\rho_0^-,\varphi_0^-,p_0^-,v_{\rm c})-\epsilon D_{v}\mathcal{P}^{(1)}w\|_{1,\alpha,\Gamma_{\rm ex}}^{(-\alpha,\partial\Gamma_{\rm ex})}\le C\epsilon^2.
\end{equation*}

{\bf 3.} (Partial Fr\'echet derivative of  $\mathcal{P}^{(2)}$)
By the definition \eqref{def-RQ} of $\mathcal{P}^{(2)}$, we have
\begin{equation*}
\mathcal{P}^{(2)}(\rho_0^-,\varphi_0^-,p_0^-,v_{\rm c}+\epsilon w)=\left.\left(\frac{\gamma-1}{\gamma }\right)^{\frac{\gamma}{\gamma-1}}S_{\epsilon}^{\frac{-1}{\gamma-1}}\right\rvert_{\Gamma_{\rm ex}}.
\end{equation*}
Differentiating this with respect to $\epsilon$ at $\epsilon=0$ gives 
\begin{equation*}
\begin{split}
\left.\frac{d}{d\epsilon}\right|_{\epsilon=0}\mathcal{P}^{(2)}(\rho_0^-,\varphi_0^-,p_0^-,v_{\rm c}+\epsilon w)
&=\left.-\frac{1}{\gamma-1}\left(\frac{\gamma-1}{\gamma }\right)^{\frac{\gamma}{\gamma-1}}S_0^{\frac{-\gamma}{\gamma-1}}{S}^{(w)}\right\rvert_{\Gamma_{\rm ex}}.
\end{split}
\end{equation*}
By \eqref{f-FD} and \eqref{S-FD}, we have
\begin{equation*}
\begin{split}
\left.\frac{d}{d\epsilon}\right|_{\epsilon=0}\mathcal{P}^{(2)}(\rho_0^-,\varphi_0^-,p_0^-,v_{\rm c}+\epsilon w)
&=-\frac{1}{\gamma-1}\left(\frac{\gamma-1}{\gamma }\right)^{\frac{\gamma}{\gamma-1}}\frac{S_0^{\frac{-\gamma}{\gamma-1}}A_{\ast}\varphi^{(w)}(r_{\rm sh},\cdot)}{\partial_r(\varphi_0^--\varphi_0^+)(r_{\rm sh})}. 
\end{split}
\end{equation*}
Then, by the estimates in \eqref{FD-est}, we have
\begin{equation*}
\|\mathcal{P}^{(2)}(\rho_0^-,\varphi_0^-,p_0^-,v_{\rm c}+\epsilon w)-\mathcal{P}^{(2)}(\rho_0^-,\varphi_0^-,p_0^-,v_{\rm c})-\epsilon D_{v}\mathcal{P}^{(2)}w\|_{1,\alpha,\Gamma_{\rm ex}}^{(-\alpha,\partial\Gamma_{\rm ex})}\le C\epsilon^2.
\end{equation*}
%
%
This finishes the proof of Lemma \ref{lemma-RQ} (i).
\end{proof}
\begin{corollary} 
$D_v\mathcal{P}^{(2)}: C^{1,\alpha}_{(-\alpha,\partial\Gamma_{\rm ex})}(\Gamma_{\rm ex})\to C^{1,\alpha}_{(-\alpha,\partial\Gamma_{\rm ex})}(\Gamma_{\rm ex})$ is a compact mapping.
\end{corollary}

\begin{proof}
For any $w\in C^{1,\alpha}_{(-\alpha,\partial\Gamma_{\rm ex})}(\Gamma_{\rm ex})$, it follows from \eqref{DQ}, \eqref{S-FD}, and \eqref{f-FD} that 
\begin{equation*}
\|D_v\mathcal{P}^{(2)}w\|_{2,\alpha,\Gamma_{\rm ex}}^{(-1-\alpha,\partial\Gamma_{\rm ex})}\le CA_{\ast}\|\varphi^{(w)}(r_{\rm sh},\cdot)\|_{2,\alpha,\Gamma_{\rm ex}}^{(-1-\alpha,\partial\Gamma_{\rm ex})}.
\end{equation*}
Using \eqref{sys-var}, \eqref{sys-W}, and \eqref{S-FD}, we get a weighted $C^{2,\alpha}$ estimate of $\varphi^{(w)}$ as follows:
\begin{equation*}
\begin{split}
\|\varphi^{(w)}\|_{2,\alpha,\Omega_{\rm sh}^+}^{(-1-\alpha,\Gamma_{\rm w})}
\le&\, C\left(\|{\bf W}^{(w)}\|_{2,\alpha,\Omega_{\rm sh}^+}^{(-1-\alpha,\Gamma_{\rm w})}+\|w\|_{1,\alpha,\Gamma_{\rm ex}}^{(-\alpha,\partial\Gamma_{\rm ex})}\right)\\
\le&\, C\left(\|S^{(w)}\|_{1,\alpha,\Omega_{\rm sh}^+}^{(-\alpha,\Gamma_{\rm w})}+\|w\|_{1,\alpha,\Gamma_{\rm ex}}^{(-\alpha,\partial\Gamma_{\rm ex})}\right)\\
\le&\, C\left(A_{\ast}\|\varphi^{(w)}\|_{2,\alpha,\Omega_{\rm sh}^+}^{(-1-\alpha,\Gamma_{\rm w})}+\|w\|_{1,\alpha,\Gamma_{\rm ex}}^{(-\alpha,\partial\Gamma_{\rm ex})}\right).
\end{split}
\end{equation*}
Since $0<A_{\ast}\le C\epsilon$ by the definition \eqref{def-Astar} of $A_{\ast}$, if $\epsilon$ is sufficiently small depending only on the data, then 
\begin{equation}\label{varphi-w-est}
\|\varphi^{(w)}\|_{2,\alpha,\Omega_{\rm sh}^+}^{(-1-\alpha,\Gamma_{\rm w})}\le C\|w\|_{1,\alpha,\Gamma_{\rm ex}}^{(-\alpha,\partial\Gamma_{\rm ex})}.
\end{equation}
Therefore we have 
\begin{equation*}
\|D_v\mathcal{P}^{(2)}w\|_{2,\alpha,\Gamma_{\rm ex}}^{(-1-\alpha,\partial\Gamma_{\rm ex})}\le C\|w\|_{1,\alpha,\Gamma_{\rm ex}}^{(-\alpha,\partial\Gamma_{\rm ex})}.
\end{equation*}
\end{proof}

\begin{proof}[Proof of Lemma \ref{P-empty} (ii)] 
One can prove it by a standard method. For details, one may refer to \cite[Proof of Lemma 5.4]{bae2011transonic}. 
\end{proof}

\begin{proof}[Proof of Lemma \ref{P-empty} (iii)] 
We prove the invertibility of $D_{v}\mathcal{P}: C^{1,\alpha}_{(-\alpha,\Gamma_{\rm ex})}(\Gamma_{\rm ex})\to C^{1,\alpha}_{(-\alpha,\Gamma_{\rm ex})}(\Gamma_{\rm ex})$.
On $\Gamma_{\rm ex}$, by \eqref{sys-var}, $\varphi^{(w)}$ satisfies 
\begin{equation*}\label{varphi-vex-eq}
k\partial_r\varphi^{(w)}=w
\end{equation*}
for a constant $k$ defined by 
\begin{equation}\label{k12-def}
k=\left(B_0-\frac{|\nabla\varphi_0^+|^2}{2}\right)^{\frac{2-\gamma}{\gamma-1}}\frac{(\gamma+1)\left(K_0-|\nabla\varphi_0^+|^2\right)}{2(\gamma-1)}>0
\end{equation}
with $K_0$ given in \eqref{K0-def}.
From this and \eqref{DR}, we have 
\begin{equation}\label{DvR}
\begin{split}
D_v\mathcal{P}^{(1)}w
&=\left.-\frac{\gamma\left(B_0-\frac{1}{2}|\nabla\varphi_0^+|^2\right)^{\frac{1}{\gamma-1}}}{\gamma-1}\nabla\varphi_0^+\cdot\nabla\varphi^{(w)}\right|_{\Gamma_{\rm ex}}
=a_1w
\end{split}
\end{equation}
for a constant $a_1$ defined by 
\begin{equation}\label{a1-sign}
a_1:=\left.-\frac{\gamma\left(B_0-\frac{1}{2}|\nabla\varphi_0^+|^2\right)^{\frac{1}{\gamma-1}}}{\gamma-1}\cdot\frac{\partial_r\varphi_0^+}{k}\right|_{\Gamma_{\rm ex}}.
\end{equation}
Since  $k>0$ by the definition of $k$ in \eqref{k12-def}, $a_1$ is strictly negative, i.e., $a_1<0$.
And, it follows from \eqref{f-FD}, \eqref{S-FD}, and \eqref{def-Astar} that
\begin{equation}\label{DvQ}
\begin{split}
D_v\mathcal{P}^{(2)}w
&=a_2\varphi^{(w)}(r_{\rm sh},\cdot)
\end{split}
\end{equation}
for a constant $a_2$ defined by 
\begin{equation}\label{a2-sign}
a_2:=\left.-\frac{1}{\gamma-1}\left(\frac{\gamma-1}{\gamma }\right)^{\frac{\gamma}{\gamma-1}}(S_0^+)^{\frac{-\gamma}{\gamma-1}}\frac{A_{\ast}}{\partial_r(\varphi_0^--\varphi_0^+)(r_{\rm sh})}\right|_{\Gamma_{\rm ex}}.
\end{equation}
Since  $A_{\ast}>0$ by the definition  of $A_{\ast}$ in \eqref{def-Astar}, the sign of $a_2$ is negative, i.e., $a_2<0$.
Then, by \eqref{DvR} and \eqref{DvQ}, we can express $D_v\mathcal{P}$  as 
\begin{equation}\label{def-DP}
D_v\mathcal{P}w=b_1(I-T)w
\end{equation}
for $b_1:=\mathcal{P}^{(2)}(\zeta_0)a_1$ and a mapping $T: C^{1,\alpha}_{(-\alpha,\partial\Gamma_{\rm ex})}(\Gamma_{\rm ex})\to C^{1,\alpha}_{(-\alpha,\partial\Gamma_{\rm ex})}(\Gamma_{\rm ex})$ given by 
\begin{equation*}
T:w\mapsto -\frac{b_2}{b_1}\varphi^{(w)}(r_{\rm sh},\cdot)\quad\mbox{for}\quad b_2:=\mathcal{P}^{(1)}(\zeta_0)a_2.
\end{equation*}
Since $T$ is compact, the Fredholm alternative theorem yields that either $D_v\mathcal{P}w=0$ has a nontrivial solution $w$ or $D_v\mathcal{P}$ is invertible. 
 By using the eigenfunction expansion of $w$ and a contraction argument, one can prove that $D_{v}\mathcal{P}w=0$ if and only if $w=0$. 
The details of the proof are same as \cite[Step 2 in the Proof of Lemma 5.5]{bae2011transonic}. So we skip it here. 
The proof of Lemma \ref{P-empty} (iii) is completed.
\end{proof}


\subsubsection{Uniqueness: Proof of Theorem \ref{Helmholtz-Theorem} (b)}\label{S-uni-Thm}

Suppose that  axisymmetric functions $v_1$, $v_2\in C^{1,\alpha}_{(-\alpha,\partial\Gamma_{\rm ex})}(\Gamma_{\rm ex})$ satisfy 
\begin{equation}\label{uni-p}
\mathcal{P}(\zeta_{\ast},v_1)=p_{\rm ex}=\mathcal{P}(\zeta_{\ast},v_2)\mbox{ for }\zeta_{\ast}:=(\rho^-,\varphi^-,p^-).
\end{equation}
We claim that $v_1=v_2$.

\begin{lemma}\label{lemma-DvP}{\cite[Lemma 6.1]{bae2011transonic}}
For $\alpha\in(\frac{1}{2},1)$, $D_v\mathcal{P}$ is a bounded linear mapping from $C^{0,\alpha}_{(1-\alpha,\partial\Gamma_{\rm ex})}(\Gamma_{\rm ex})$ to itself and $D_v\mathcal{P}$ is invertible in $C^{0,\alpha}_{(1-\alpha,\partial\Gamma_{\rm ex})}(\Gamma_{\rm ex})$. 
\end{lemma}

Lemma \ref{lemma-DvP} and \eqref{uni-p} yield that
\begin{equation}\label{v12-DvP}
v_1-v_2=-D_v\mathcal{P}^{-1}(\mathcal{P}(\zeta_{\ast},v_1)-\mathcal{P}(\zeta_{\ast},v_2)-D_v\mathcal{P}(v_1-v_2)).
\end{equation}

\begin{lemma}\label{lemma-est-v12}
Fix $\alpha\in(\frac{1}{2},1)$ and  set $\mathfrak{e}_k:=\|v_k-v_c\|_{1,\alpha,\Gamma_{\rm ex}}^{(-\alpha,\partial\Gamma_{\rm ex})}$ for $k=1,2$. 
Then there exists a constant $\sigma_u\in(0,\sigma_2]$  depending only on the data so that if
\begin{equation*}
\sigma_{\natural}:=\|(\rho^-,{\bf u}^-,p^-)-(\rho_0^-,\nabla\varphi_0^-,p_0^-)\|_{2,\alpha,\Om}+\mathfrak{e}_1+\mathfrak{e}_2\le \sigma_u,
\end{equation*}
then 
\begin{equation}\label{P12-est}
\begin{split}
&\|\mathcal{P}(\zeta_{\ast},v_1)-\mathcal{P}(\zeta_{\ast},v_2)-D_v\mathcal{P}(v_1-v_2)\|_{0,\alpha,\Gamma_{\rm ex}}^{(1-\alpha,\partial\Gamma_{\rm ex})}
\le C\sigma_{\natural}\|v_1-v_2\|_{0,\alpha,\Gamma_{\rm ex}}^{(1-\alpha,\partial\Gamma_{\rm ex})}.
\end{split}
\end{equation}
\end{lemma}
\begin{proof}
 According to  the proof of Theorem \ref{Helmholtz-Theorem} (a), for each $k=1,2$, if $\sigma_{\natural}\le \sigma_2$, then there exists an axisymmetric solution $(f_k,S_k,\Lambda_k,\chi_k+\varphi_0^+,{\bf W}_k)$ to the free boundary problem \eqref{E-re} with \eqref{bd-W}-\eqref{bd-EX} and \eqref{boundary-HD} associated with $(\zeta_{\ast}, v_k)$. Moreover, each solution satisfies the estimate 
\begin{equation}\label{each-est}
\begin{split}
\|f_k-r_{\rm sh}\|_{2,\alpha,(0,\phi_0)}^{(-1-\alpha,\{\phi=\phi_0\})}
+\|(S_k,\frac{\Lambda_k}{r\sin\phi}{\bf e}_{\theta})-(S_0^+,{\bf 0})\|_{1,\alpha,\Om_{f_k}^+}^{(-\alpha,\Gamma_{\rm w})}&\\
+\|\varphi_k-\varphi_0^+\|_{2,\alpha,\Om_{f_k}^+}^{(-1-\alpha,\Gamma_{\rm w})}+\|\psi_k{\bf e}_{\theta}\|_{2,\alpha,\Om_{f_k}^+}^{(-1-\alpha,\Gamma_{\rm w})}&\le C\sigma_{\natural}.
\end{split}
\end{equation}

For notational simplicity, let us set
\begin{equation*}
\begin{split}
&(\tilde{S}_k,\tilde{\Lambda}_k,\tilde{\chi}_k,\tilde{\bf W}_k):=(S_k,\Lambda_k,\chi_k,{\bf W}_k)\circ\mathfrak{T}_{r_{\rm sh},f_k},\\
&dS:=\tilde{S}_1-\tilde{S}_2,\quad d{\bf \Lambda}:=\frac{\tilde{\Lambda}_1}{r\sin\phi}{\bf e}_{\theta}-\frac{\tilde{\Lambda}_2}{r\sin\phi}{\bf e}_{\theta},\\
&d\chi:=\tilde{\chi}_1-\tilde{\chi}_2,\quad d{\bf W}:=\tilde{\bf W}_1-\tilde{\bf W}_2.
\end{split}
\end{equation*}
%
%
As in the same way of the proof of Proposition \ref{free-v-Pro}, 
one can check that  
there exists a constant $\sigma_{u}\in(0,\sigma_2]$ depending only on the data so that if $\sigma_{\natural}\le\sigma_{u}$, then 
\begin{equation}\label{c1-est}
\begin{split}
&\|f_1-f_2\|_{1,\alpha,(0,\phi_0)}^{(-\alpha,\{\phi=\phi_0\})}+\|dS\|_{0,\alpha,\Omega_{r_{\rm sh}}^+}^{(1-\alpha,\Gamma_{\rm w})}
+\|d\chi\|_{1,\alpha,\Omega_{r_{\rm sh}}^+}^{(-\alpha,\Gamma_{\rm w})}+\|d{\bf W}\|_{1,\alpha,\Omega_{r_{\rm sh}}^+}^{(-\alpha,\Gamma_{\rm w})}\\
&\le C\|v_1-v_2\|_{0,\alpha,\Gamma_{\rm ex}}^{(1-\alpha,\partial\Gamma_{\rm ex})}
\end{split}
\end{equation}
and
\begin{equation}\label{l-c1-est}
\|d{\bf \Lambda}\|_{0,\alpha,\Omega_{r_{\rm sh}}^+}^{(1-\alpha,\Gamma_{\rm w})}\le C\sigma_{\natural}\|v_1-v_2\|_{0,\alpha,\Gamma_{\rm ex}}^{(1-\alpha,\partial\Gamma_{\rm ex})}.
\end{equation}


Let $(f^{(w)}, S^{(w)},\varphi^{(w)},{\bf W}^{(w)})$ be from Lemma \ref{lemma-RQ} with $w=v_1-v_2$.
Then, $(\varphi^{(w)},{\bf W}^{(w)})$ is the unique solution of \eqref{sys-var}-\eqref{sys-W} with $w=v_1-v_2$, and, by \eqref{f-FD} and \eqref{S-FD}, 
\begin{equation}\label{fw-def}
f^{(w)}=\frac{\varphi^{(w)}(r_{\rm sh},\phi)}{\partial_r(\varphi_0^--\varphi_0^+)(r_{\rm sh})}\quad\mbox{and}\quad S^{(w)}=A_{\ast}f^{(w)}.
\end{equation}
Moreover, by \eqref{varphi-w-est} and \eqref{fw-def}, $(f^{(w)}, S^{(w)},\varphi^{(w)},{\bf W}^{(w)})$ satisfies 
\begin{equation}\label{ww-est-v12}
\begin{split}
&\|f^{(w)}\|_{2,\alpha,(0,\phi_0)}^{(-1-\alpha,\{\phi=\phi_0\})}+\|S^{(w)}\|_{2,\alpha,\Omega_{\rm sh}^+}^{(-1-\alpha,\Gamma_{\rm w})}\\
&+\|\varphi^{(w)}\|_{2,\alpha,\Omega_{\rm sh}^+}^{(-1-\alpha,\Gamma_{\rm w})}+\|{\bf W}^{(w)}\|_{2,\alpha,\Omega_{\rm sh}^+}^{(-1-\alpha,\Gamma_{\rm w})}
\le C\|v_1-v_2\|_{1,\alpha,\Gamma_{\rm ex}}^{(-\alpha,\partial\Gamma_{\rm ex})}.
\end{split}
\end{equation}
Then, we again use the method of the proof of Proposition \ref{free-v-Pro} to get 
\begin{equation}\label{c2-est}
\begin{split}
&\|\tilde{\chi}_1-\tilde{\chi}_2-\varphi^{(w)}\|_{1,\alpha,\Omega_{r_{\rm sh}}^+}^{(-\alpha,\Gamma_{\rm w})}\le C\sigma_{\natural}\|v_1-v_2\|_{0,\alpha,\Gamma_{\rm ex}}^{(1-\alpha,\partial\Gamma_{\rm ex})},\\
&\|\tilde{\bf W}_1-\tilde{\bf W}_2-{\bf W}^{(w)}\|_{1,\alpha,\Omega_{r_{\rm sh}}^+}^{(-\alpha,\Gamma_{\rm w})}\le C\sigma_{\natural}\|v_1-v_2\|_{0,\alpha,\Gamma_{\rm ex}}^{(1-\alpha,\partial\Gamma_{\rm ex})},\\
&\|\tilde{S}_1-\tilde{S}_2-S^{(w)}\|_{0,\alpha,\Omega_{r_{\rm sh}}^+}^{(1-\alpha,\Gamma_{\rm w})}\le C\sigma_{\natural}\|v_1-v_2\|_{0,\alpha,\Gamma_{\rm ex}}^{(1-\alpha,\partial\Gamma_{\rm ex})},\\
&\|f_1-f_2-f^{(w)}\|_{1,\alpha,(0,\phi_0)}^{(-\alpha,\{\phi=\phi_0\})}\le C\sigma_{\natural}\|v_1-v_2\|_{0,\alpha,\Gamma_{\rm ex}}^{(1-\alpha,\partial\Gamma_{\rm ex})}.
\end{split}
\end{equation}
The estimate \eqref{P12-est} directly follows from \eqref{each-est}, \eqref{c1-est}-\eqref{l-c1-est}, and \eqref{ww-est-v12}-\eqref{c2-est}.
%
%
%
The proof of Lemma \ref{lemma-est-v12} is completed.
\end{proof}


By \eqref{vex-pex}, it holds that
$\sigma_{\natural}\le C\sigma(\rho^-,{\bf u}^-,p^-,p_{\rm ex}).$
Then, from \eqref{v12-DvP} and Lemma \ref{lemma-est-v12}, we have
\begin{equation}\label{v12-uni}
\|v_1-v_2\|_{0,\alpha,\Gamma_{\rm ex}}^{(1-\alpha,\partial\Gamma_{\rm ex})}\le C^{\dagger}\sigma(\rho^-,{\bf u}^-,p^-,p_{\rm ex}) \|v_1-v_2\|_{0,\alpha,\Gamma_{\rm ex}}^{(1-\alpha,\partial\Gamma_{\rm ex})}
\end{equation}
for a constant $C^{\dagger}$ depending only on the data.
We finally choose $\tilde{\sigma}_2$ satisfying
\begin{equation*}
0<\tilde{\sigma}_2\le \min\left\{\sigma_u, \frac{1}{2 C^{\dagger}}\right\}
\end{equation*}
for $\sigma_u\in(0,\sigma_2]$ given in Lemma \ref{lemma-est-v12}
so that \eqref{v12-uni} implies $v_1=v_2$.
The proof of Theorem \ref{Helmholtz-Theorem} (b) is completed. \qed

\vspace{.25in}
\noindent

{\bf Acknowledgments}\quad
The research of Hyangdong Park was supported in part by the POSCO Science Fellowship of POSCO TJ Park Foundation and a KIAS Individual Grant (MG086701, MG086702) at Korea Institute for Advanced Study.

\bigskip
{\bf Conflict of interests}\quad
There is no conflict of interest.

\bigskip

{\bf Data availability statement}\quad 
Data sharing not applicable to this article as no datasets were generated or analyzed during the current study.

\bigskip
\bibliographystyle{siam}
\bibliography{References}

\begin{thebibliography}{10}

\bibitem{Bae:2021aa}
{\sc M.~Bae, B.~Duan, J.~Xiao, and C.~Xie}, {\em Structural stability of
  supersonic solutions to the euler--poisson system}, Archive for Rational
  Mechanics and Analysis, 239 (2021), pp.~679--731.

\bibitem{bae2014subsonic}
{\sc M.~Bae, B.~Duan, and C.~Xie}, {\em Subsonic solutions for steady
  euler--poisson system in two-dimensional nozzles}, SIAM Journal on
  Mathematical Analysis, 46 (2014), pp.~3455--3480.

\bibitem{bae2011transonic}
{\sc M.~Bae and M.~Feldman}, {\em Transonic shocks in multidimensional
  divergent nozzles}, Archive for rational mechanics and analysis, 201 (2011),
  pp.~777--840.

\bibitem{bae2019contact}
{\sc M.~Bae and H.~Park}, {\em Contact discontinuities for 2-dimensional
  inviscid compressible flows in infinitely long nozzles}, SIAM Journal on
  Mathematical Analysis, 51 (2019), pp.~1730--1760.

\bibitem{bae2019contact3D}
\leavevmode\vrule height 2pt depth -1.6pt width 23pt, {\em Contact
  discontinuities for 3-d axisymmetric inviscid compressible flows in
  infinitely long cylinders}, Journal of Differential Equations, 267 (2019),
  pp.~2824--2873.

\bibitem{bae20183}
{\sc M.~Bae and S.~Weng}, {\em 3-d axisymmetric subsonic flows with nonzero
  swirl for the compressible euler--poisson system}, in Annales de l'Institut
  Henri Poincar{\'e} C, Analyse non lin{\'e}aire, vol.~35, Elsevier, 2018,
  pp.~161--186.

\bibitem{chen2003multidimensional}
{\sc G.-Q. Chen and M.~Feldman}, {\em Multidimensional transonic shocks and
  free boundary problems for nonlinear equations of mixed type}, Journal of the
  American Mathematical Society, 16 (2003), pp.~461--494.

\bibitem{chen2007existence}
\leavevmode\vrule height 2pt depth -1.6pt width 23pt, {\em Existence and
  stability of multidimensional transonic flows through an infinite nozzle of
  arbitrary cross-sections}, Archive for Rational Mechanics and Analysis, 184
  (2007), pp.~185--242.

\bibitem{chen2018mathematics}
{\sc G.-Q.~G. Chen and M.~Feldman}, {\em The Mathematics of Shock
  Reflection-diffraction and Von Neumann's Conjectures}, vol.~359, Princeton
  University Press, 2018.

\bibitem{Chen2022}
{\sc G.-Q.~G. Chen and M.~Feldman}, {\em Multidimensional transonic shock waves
  and free boundary problems}, Bulletin of Mathematical Sciences, 12 (2022),
  p.~2230002.

\bibitem{courant1999supersonic}
{\sc R.~Courant and K.~O. Friedrichs}, {\em Supersonic flow and shock waves},
  vol.~21, Springer Science \& Business Media, 1999.

\bibitem{FANG202162}
{\sc B.~Fang and X.~Gao}, {\em On admissible positions of transonic shocks for
  steady euler flows in a 3-d axisymmetric cylindrical nozzle}, Journal of
  Differential Equations, 288 (2021), pp.~62--117.

\bibitem{Fang2021}
{\sc B.~Fang and Z.~Xin}, {\em On admissible locations of transonic shock
  fronts for steady euler flows in an almost flat finite nozzle with prescribed
  receiver pressure}, Communications on Pure and Applied Mathematics, 74
  (2021), pp.~1493--1544.

\bibitem{han2011elliptic}
{\sc Q.~Han and F.~Lin}, {\em Elliptic partial differential equations}, vol.~1,
  American Mathematical Soc., 2011.

\bibitem{Huang:2021aa}
{\sc F.~Huang, J.~Kuang, D.~Wang, and W.~Xiang}, {\em Stability of transonic
  contact discontinuity for two-dimensional steady compressible euler flows in
  a finitely long nozzle}, Annals of PDE, 7 (2021), p.~23.

\bibitem{park2020transonic}
{\sc H.~Park and H.~Ryu}, {\em Transonic shocks for 3-d axisymmetric
  compressible inviscid flows in cylinders}, Journal of Differential Equations,
  269 (2020), pp.~7326--7355.

\bibitem{park20213}
{\sc Y.~Park}, {\em 3-d axisymmetric transonic shock solutions of the full
  euler system in divergent nozzles}, Archive for Rational Mechanics and
  Analysis, 240 (2021), pp.~467--563.

\bibitem{weng2021structural}
{\sc S.~Weng, C.~Xie, and Z.~Xin}, {\em Structural stability of the transonic
  shock problem in a divergent three-dimensional axisymmetric perturbed
  nozzle}, SIAM Journal on Mathematical Analysis, 53 (2021), pp.~279--308.

\bibitem{xin2009transonic}
{\sc Z.~Xin, W.~Yan, and H.~Yin}, {\em Transonic shock problem for the euler
  system in a nozzle}, Archive for rational mechanics and analysis, 194 (2009),
  pp.~1--47.

\bibitem{Yuan2008remark}
{\sc H.~Yuan}, {\em A remark on determination of transonic shocks in divergent
  nozzles for steady compressible euler flows}, Nonlinear Analysis: Real World
  Applications, 9 (2008), pp.~316--325.

\end{thebibliography}

\end{document}